\documentclass[11pt]{amsart}

\usepackage{amssymb,amsfonts,amsthm,amsmath}
\usepackage[english]{babel}
\usepackage[pdftex]{graphicx}
\usepackage[colorlinks=true, pdfstartview=FitV, linkcolor=blue, citecolor=blue, urlcolor=blue]{hyperref}
\usepackage[T1]{fontenc}
\usepackage[utf8x]{inputenc}
\usepackage{fancyhdr}
\usepackage{graphicx}
\usepackage{float}
\usepackage{color}
\usepackage{subfigure}
\usepackage{xspace}
\usepackage{geometry}
\usepackage{tikz-cd}
\usepackage[font=small,skip=0pt]{caption}

\newtheorem{theorem}{Theorem}[section]

\newtheorem{coro}[theorem]{Corollary}

\theoremstyle{definition}

\theoremstyle{remark}

\newtheorem{remark}[theorem]{Remark}

\newcommand{\R}{\mathbb{R}}  

\begin{document}


\title[Addition Theorems for $\mathcal{C}^k$ real functions and applications]{Addition Theorems for $\mathcal{C}^k$ real functions and applications in Ordinary Differential Equations}


\author[F. Crespo]{Francisco Crespo}
\address{GISDA, Departamento de Matem\'atica, Facultad de Ciencias, Universidad del B\'\i o-B\'\i o, Av. Collao 1202, Casilla 5-C. Concepci\'on,  Chile}
\email{fcrespo@ubiobio.cl}

\author[S. Rebollo-Perdomo]{Salom\'on Rebollo-Perdomo}
\address{GISDA, Departamento de Matem\'atica, Facultad de Ciencias, Universidad del B\'\i o-B\'\i o, Av. Collao 1202, Casilla 5-C. Concepci\'on,  Chile}
\email{srebollo@ubiobio.cl}

\author[J.L. Zapata]{Jorge L. Zapata}
\address{GISDA, Departamento de Matem\'atica, Facultad de Ciencias, Universidad del B\'\i o-B\'\i o, Av. Collao 1202, Casilla 5-C. Concepci\'on,  Chile}
\email{jzapata@ubiobio.cl}


\begin{abstract}
This work establishes the existence of addition theorems and double-angle formulas for $\mathcal{C}^k$ real scalar functions. Moreover, we determine necessary and sufficient conditions for a bivariate function to be an addition formula for a $\mathcal{C}^k$ real function. The double-angle formulas allow us to generate a duplication algorithm, which can be used as an alternative to the classical numerical methods to obtain an  approximation for the solution of an ordinary differential equation.  We demonstrate that this algorithm converges uniformly in any compact  domain contained in the maximal domain of that solution. Finally, we carry out some numerical simulations showing a good performance of the duplication algorithm when compared with standard numerical methods.

\end{abstract}





\maketitle

\section{Introduction}
We recall that a real function $\phi(t)$  has an addition theorem  if  $\phi(t+\tau)$ may be recovered from $\phi(t)$ and $\phi(\tau)$ through the following relation
\begin{equation}
\label{eq:AdditionTheoremFormula}
G(\phi(t),\phi(\tau),\phi(t+\tau))=0,
\end{equation}
wherever the domains of $G(x,y,z)$ and $\phi(t)$ are defined. 
Hereinafter, the function $G(x,y,z)$ defining the addition theorem will be dubbed as the addition formula associated to the function $\phi(t)$. It is well known that trigonometric or exponential functions are examples of functions endowed with these features. For instance, the exponential function $\phi(t)=e^{at}$, with 
$a \in \R\setminus\{0\}$, has the addition formula
$$
G(x,y,z)=z-xy=0,
$$ 
because
$e^{a(t+\tau)}-e^{at}e^{a\tau}=0$  for all $t,\tau \in \R$. 

Traditionally, the theory of elliptic functions investigates the existence of addition formulas \cite{Hancock,MontelPaul,PaulPainleve,Ritt}. Indeed, quoting Weierstrass \cite{Weierstrass}, we have that: \emph{"the problem of the theory of elliptic functions is to determine all functions of the complex argument for which there exists an algebraic addition theorem''}. In this regard, the following statement is proved in \cite{Phragmen1885}. Every single-valued analytic function with an algebraic addition theorem is either an elliptic function or an algebraic function of $x$ or $e^{ax}$. 

Equation \eqref{eq:AdditionTheoremFormula} can be denominated as a scalar addition theorem because it depends only on one function. However, this equation can be generalized and consider an equation of the form
$$
\mathcal{G}(\phi_1(t),\phi_2(\tau),\phi_3(t+\tau))=0,
$$
where 
$\phi_i(t)\in\mathbb{R}^N$ for $i=1,2,3$, which are called vector addition theorems \cite{Bukhshtaber1993}.
In \cite{Bukhshtaber1993,Bukhshtaber1996} the authors related these kind of functional equations with the theory of one-dimensional integrable systems.  Although in  this paper we will restrict to the case of scalar addition theorems \eqref{eq:AdditionTheoremFormula} we consider this investigation as the basis for the realization of a project that includes the case of vector addition theorems from the point of view of differential equations, which as we will show in this work could provide different perspectives for applications.

In the theory of addition theorems we distinguish between direct and inverse problems. More precisely, the problem of finding an addition theorem for a given function $\phi(t)$ will be called the \emph{direct problem of addition theorems}. Conversely, one may also rise the question of whether a given function $G$ could be an addition formula, which is referred to as \emph{inverse problem of addition theorems}.

Although we have considered in \eqref{eq:AdditionTheoremFormula} the general definition of an addition formula, we will focus throughout this paper on \emph{explicit} addition theorems or formulas. That is to say, we call a bivariate function $R(x,y)$ an \emph{explicit addition formula} for $\phi(t)$ when relation \eqref{eq:AdditionTheoremFormula} may be expressed as
\begin{equation}
\label{eq:RelationAdditionExplicit}
\phi(t+\tau)=R\left(\phi(t),\phi(\tau)\right),
\end{equation}
wherever the domains of the involved functions $R(x,y)$ and $\phi(t)$ are defined. Therefore, we are interested in the direct and inverse problems for explicit addition formulas.

The interest in the solutions and properties of explicit addition theorems is a classical problem. For instance, by fixing $R(x,y)$ as a polynomial, rational or algebraic type, the existence and uniqueness of solutions for  equations of the form \eqref{eq:RelationAdditionExplicit} are discussed in \cite[Sec. 2.2.4]{Aczel1967}, and references therein, where a treatment of addition theorems is widely carried out from the point of view of functional equations.

The existence of addition theorems and double-angle formulas, which are the particular case of considering $t=\tau$ in \eqref{eq:AdditionTheoremFormula} or \eqref{eq:RelationAdditionExplicit}, is of great practical interest. For instance, they are used in the construction of algebraic  first integrals for differential equations. This fact is pointed out by Tsiganov \cite{Tsiganov2009}, where the author recalls that the integrals of the Kepler system, given by the  eccentricity and  the angular momentum vectors, are particular cases of this procedure. 
Moreover, before the age of computers, double-angle formulas were traditionally  used to compute tables of trigonometric,  exponential or logarithm functions.  Still today, this strategy is widely used in the fast and efficient computation of elliptic functions and integrals, see for example \cite{Bulirsch1965,Bulirsch1969,Carlson1979,Fukushima2015} and the references therein. Since these works relies on the use of double-angle formulas,  their applicability  could be extended to any function having a double-angle formula.

In this paper we deal with the problem of finding addition and double-angle theorems for $\mathcal{C}^k$ real functions, with $k\in \mathbb{N}\cup \{\infty,\omega\}$, where $\mathcal{C}^{\omega}$ stands for analytic functions. In the applications we pay special attention to functions coming from solutions of autonomous ordinary differential equations.  This problem was  previously explored in \cite{ThesisZapata}.  In  \cite{AbrahamUngar1983} the case of linear homogeneous ordinary differential equations with constant coefficients was analyzed, and also in \cite{AbrahamUngar1987}, where the author places the binomial theorem and the addition theorems for exponential, trigonometric, and hyperbolic functions in the context of a single addition theorem generated by an initial value problem. Our approach allows to consider any autonomous  ordinary differential equations of class  $\mathcal{C}^k$.

As an application of the double-angle formulas, we propose an alternative numerical approximation of the solution of an initial value problem based on the duplication algorithm. This scheme requires the double-angle formula of the solution to be available, which restricts the applicability of the numerical method. Nevertheless, we will use the Taylor expansion of double-angle formulas associated to the solution of an initial value problem. Therefore, extending the applicability of the duplication algorithm and allowing to generate a polygonal approximation of the named solutions. In addition, we also establish the convergence of the proposed duplication algorithm.

This paper is organized as follows. In Section~\ref{sec:FundamentalTheory}, we deal with results concerning the direct and inverse theory of addition theorems. Section~\ref{sec:Computation} contains the Taylor expansion for double-angle formulas associated to the solution of an initial value problem, which allows to define a duplication algorithm in Section~\ref{Algorithm}. Finally, in Section~\ref{sec:Simulations} we carry out numerical experiments comparing the duplication algorithm with the standard numerical methods provided by  \emph{Wolfram Mathematica}.


\section{On the Theory of Addition Theorems}
\label{sec:FundamentalTheory}
This section presents some results about addition theorems as existence or necessary and sufficient conditions for a function $R(x,y)$ to be an explicit addition formula.

\subsection{Direct Problem of Addition Theorems}
\label{sec:DirectInverse}
The following theorem gives an affirmative answer to the question of whether exists addition theorems for a given function. This problem was tackled in  \cite{Aczel1967}, see page 256, as a functional equation of the form \eqref{eq:RelationAdditionExplicit}. Here we specialized  to the case of $\mathcal{C}^k$ functions, which allows for a shortened proof.

Previous to the existence theorem, we need some auxiliary notation and results. Let $I\subset \mathbb{R}$ is an open interval.
If $\phi: I \longrightarrow \mathbb{R}$ is a $\mathcal{C}^k$  function with non vanishing derivative on $I$, then the inverse function theorem says that $\phi:I\longrightarrow \Delta$,
where $\Delta:=\phi({I})\subset \mathbb{R}$, is a diffeomorphism with a $\mathcal{C}^k$ inverse function $\phi^{-1}:\Delta\longrightarrow {I}$ satisfying that $\phi^{-1}\circ \phi=Id_{{I}}$ and $ \phi\circ \phi^{-1}=Id_\Delta$. If we assume that the set  $\{ ( t,\tau ) \in I \times I \; |\; t +\tau \in I\}$ is non empty,  we can define the non empty set
\begin{equation}
\label{eq:Dominio}
\mathcal{D}_{I}:=\left\{(\xi,\eta)\in\Delta\times\Delta \;|\; \phi^{-1}(\xi)+\phi^{-1}(\eta)\in I \right\} \subset\Delta\times\Delta.
\end{equation}
We note that the function
$$\chi:\Delta\times\Delta\longrightarrow\mathbb{R},\quad (\xi,\eta)\longmapsto \phi^{-1}(\xi)+\phi^{-1}(\eta),$$
is of class $\mathcal{C}^k$ and that $\mathcal{D}_{I}=\chi^{-1}(I)$, which implies that $\mathcal{D}_{I}$ is an open set.
Therefore, in what follows $\mathcal{D}_I$ will be named as the \emph{addition domain} of $\phi(t)$.

Note that  if a given function $\phi(t)$ is endowed with an explicit addition formula $R(x,y)$, the above set $\mathcal{D}_{I}$ is the maximal domain in which $R(x,y)$ can be defined.

\begin{theorem}
\label{theo:AdditionTheorem}
Let $I$ be a real open interval. If $\phi:I \longrightarrow \mathbb{R}$ is a $\mathcal{C}^k$ function with non empty addition domain $\mathcal{D}_{I}$ and with non vanishing derivative on $I$, then $\phi(t)$ is endowed with the explicit addition theorem
$$
R: \mathcal{D}_{I} \longrightarrow \mathbb{R}, 
\quad (x,y) \longmapsto \phi\big(\phi^{-1}(x)+\phi^{-1}(y)\big).  
$$ 
\end{theorem}

\begin{proof}
We define the functions 
$$
\Psi:\mathcal{D}_{I} \longrightarrow I\times I, 
\quad (x,y) \longmapsto \left(\phi^{-1}(x),\phi^{-1}(y)\right)
$$
and
$$
\Phi:\Psi(\mathcal{D}_{I}) \longrightarrow \mathbb{R}, 
\quad (t,\tau) \longmapsto \phi(t+\tau).
$$
Both functions are well-defined because of the hypothesis and they are of class $\mathcal{C}^k$. Hence, we have the following commutative diagram
\begin{center}
\begin{tikzcd}[row sep=5ex, column sep=10ex]
\mathcal{D}_{I} \arrow[rd,swap, "\Phi\circ\Psi"] \arrow[r, "\Psi"] & \Psi(\mathcal{D}_{I}) \arrow[d]\arrow[d, "\Phi"]  \\
& \mathbb{R}.
\end{tikzcd}
\end{center}
Consider the $\mathcal{C}^k$ function 
$R:\mathcal{D}_{I} \longrightarrow \mathbb{R}$ defined as
$${R}(x,y):=(\Phi\circ \Psi)(x,y).$$
Then, for all $t,\tau,t+\tau\in {I}$, we have that $\phi(t),\phi(\tau)\in\mathcal{D}_{I}$. Thus,

$$
\begin{aligned}
{R}(\phi(t),\phi(\tau))&=\Phi\circ \Psi(\phi(t),\phi(\tau)),\\
&=\Phi\big(\phi^{-1}(\phi(t)),\phi^{-1}(\phi(\tau))\big),\\
&=\Phi(t,\tau),\\
&=\phi(t+\tau).\\
\end{aligned}
$$

This proves that $R(x,y)$ is a  $\mathcal{C}^k$  explicit addition theorem for $\phi(t)$. Finally, from the definition of $\Psi$ and $\Phi$ we have
$$
R(x,y)=(\Phi\circ \Psi)(x,y)=\phi\big(\phi^{-1}(x)+\phi^{-1}(y)\big).
$$
The proof is completed.
\end{proof}

\begin{remark}
Note that, for the case $I=(a,b)$ with $a,b>0$ and $a\geq b/2$ the domain $\mathcal{D}_{I}$ is  empty. An  analogous situation occurs for $a,b<0$. Therefore, thinking  in the applications, we will consider the case of $I$ being a neighborhood   of the origin, which guarantees $\mathcal{D}_{I}\neq \emptyset$. 
\end{remark}

\begin{remark}
Theorem \eqref{theo:AdditionTheorem}  gives  the  existence of a $\mathcal{C}^k$ addition formula wherever it  makes sense to add $t,\tau\in I$. In this regard we claim that it is not local but a global result for the case of strictly monotone  functions.
\end{remark}

\subsection{Inverse Problem of Addition Theorems}
\label{sec:Inverse}
Along this section we study necessary and sufficient conditions for a function $R(x,y)$ to be an explicit addition formula for a $\mathcal{C}^k$ function.

\begin{theorem}[Necessary Conditions]
\label{CondNece}
\index{Addition Formula Necessary Conditions}
Given an interval $I$ containing the origin, a real $\mathcal{C}^k$ function $\phi:I\longrightarrow \mathbb{R}$, with non empty addition domain $\mathcal{D}_{I}$ and $\phi(0)=\phi_0$, whose derivative vanishes only in a discrete subset of $I$, and its addition formula $R(x,y):\mathcal{D}_{I}\longrightarrow \mathbb{R}$, the following statements hold:

\begin{enumerate}
\item[(i)] ${R}$ is symmetric.
\item[(ii)] ${R}(\phi_0,y)=y,\;{R}(x,\phi_0)=x.$ 
\item[(iii)] ${R}_x(\phi_0,z)=R_y(z,\phi_0).$
\item[(iv)] ${R}[{R}(x,y),z]={R}[x,{R}(y,z)]$.
\end{enumerate}
\end{theorem}
\begin{proof}
By hypothesis we have the explicit 
addition theorem \eqref{eq:RelationAdditionExplicit}, that is,
$$
\phi(t+\tau)=R(\phi(t),\phi(\tau)).
$$
Hence, the symmetry of $R(x,y)$ is obtained from the fact that  $\phi(t+\tau)=\phi(\tau+t).$ Additionally,  for all $t,\tau \in  I$ we have
$$
\phi(t+0)={R}(\phi(t),\phi_0) 
\quad
\mbox{and}
\quad
\phi(0+\tau)=R(\phi_0,\phi(\tau)).
$$
Thus, statement (ii) follows by using $\phi(t)=\phi(t+0)=x$ 
and $\phi(\tau)=\phi(0+\tau)=y$.  
 
We now differentiate \eqref{eq:RelationAdditionExplicit} with respect to $t$ and $\tau$ to obtain 
\begin{equation}
\label{eq:Parcial-1}
\dot{\phi}(t+\tau)={R}_x(\phi(t),\phi(\tau))\dot{\phi}(t)
\end{equation}
and
\begin{equation}
\label{eq:Parcial-2}
 \dot{\phi}(t+\tau)={R}_y(\phi(t),\phi(\tau))\dot{\phi}(\tau),
\end{equation}
respectively.  Equation \eqref{eq:Parcial-1} at $t=0$ yields
$$
\dot{\phi}(\tau)={R}_x(\phi_0,\phi(\tau))\dot{\phi}(0)
$$
and equation \eqref{eq:Parcial-2} with $\tau=0$ becomes
$$
\dot{\phi}(t)={R}_y(\phi(t),\phi_0)\dot{\phi}(0).
$$
From these two previous equations, and by considering $t=\tau$ in the last one, we obtain
$$
{R}_x(\phi_0,\phi(\tau))=R_y(\phi(\tau),\phi_0)
$$ 
Therefore, by putting $z=\phi(\tau)$ we get statement (iii).

Finally, statement (iv) is already given in \cite{Aczel1967}. We take $t_1,t_2,t_3\in I$ and  such that $x=\phi(t_1)$, $y=\phi(t_2)$ and $z=\phi(t_3)$  are in $\mathcal{D}_{I}$. Then, the statement follows from the following equalities
\begin{eqnarray*}
\phi(t_1+(t_2+t_3))&=&{R}(\phi(t_1),\phi(t_2+t_3))={R}\left[\phi(t_1),{R}(\phi(t_2),\phi(t_3))\right]
\end{eqnarray*}
and
\begin{eqnarray*}
\phi((t_1+t_2)+t_3)&=&{R}(\phi(t_1+t_2),\phi(t_3))={R}\left[{R}(\phi(t_1),\phi(t_2)),\phi(t_3)\right].
\end{eqnarray*}
\end{proof}

In \cite{Kuwagaki1953}, functions $R(x,y)$ satisfying condition (iv) of previous theorem are analyzed, showing that they must also satisfy condition (i). Moreover, if such a function satisfies that $R(0,0)=0$, we have that it is given by one of the following possibilities
$$R(x,y)=x,\quad R(x,y)=y,\quad R(x,y)=x+y+xy S_1(x,y),\quad R(x,y)=xy S_2(x,y),$$
where $S_1$ and $S_2$ are symmetric functions.

Next result follows directly from Theorem~\ref{CondNece}, particularly from  \eqref{eq:Parcial-1} and  \eqref{eq:Parcial-2}. It says that a differentiable function $\phi(t)$ having an explicit addition formula is also the solution of a particular autonomous first order ordinary differential equation.

\begin{coro}
\label{Corolario-EDO}
Let $I\subset \mathbb{R}$ be an open interval containing 
the origin and let $\phi:I\longrightarrow \mathbb{R}$ be a differentiable function with non empty addition domain $\mathcal{D}_{I}$ and $\phi(0)=\phi_0$.
If $\phi(t)$ has a explicit addition formula $R(x,y)$ of class  
$\mathcal{C}^k$, $k\geq 2$, defined in $\mathcal{D}_{I}$, then it is verified that
\begin{equation}
\label{eq:PDE}
{R}_{x}(\phi(t),\phi(\tau))\dot\phi(t)={R}_{y}(\phi(t),\phi(\tau))\dot\phi(\tau).
\end{equation}
Moreover, $\phi(t)$ is the solution of  the following initial value problem
$$
\dot{\phi}=\,R_x(\phi_0,\phi)\dot{\phi}(0)=R_y(\phi,\phi_0)\dot{\phi}(0), \quad \phi(0)=\phi_0.
$$
\end{coro}

\begin{remark}
The function $R(x,y)$ may be found as the solution of the partial differential equation given by \eqref{eq:PDE} with boundary conditions ${R}(x,x_0)=x,\:{R}(x_0,y)=y,$ where we have used $x=\phi(t),$ $x_0=\phi_0$ and $y=\phi(\tau)$.
\end{remark}

\begin{theorem}[Sufficient Conditions\index{Sufficient Conditions}]
\label{CondiSufi}
Let ${R}(x,y)$ be a $\mathcal{C}^k$,  $k\geq 2$, symmetric function satisfying that there exists a scalar $x_0\in\mathbb{R}$ such that 
\begin{enumerate}
\item[(i)] ${R}(x_0,y)=y,\;{R}(x,x_0)=x.$
\item[(ii)] ${R}_x(x_0,z)=R_y(z,x_0)=f(z;x_0).$
\item[(iii)] ${R}_x(a,b)\, {R}_x(x_0,a)={R}_x(x_0,{R}(a,b))
\quad \mbox{and} \quad
{R}_y(a,b)\, {R}_y(b,x_0)={R}_y({R}(a,b),x_0).$
\end{enumerate}
Then, $R(x,y)$ is the explicit addition formula for the solution of the initial value problem
\begin{equation}
\label{edo1}
\dot{x}=\,f(x;x_0),\quad x(0)=x_0.
\end{equation}
\end{theorem}

\begin{proof}
For notational simplicity, we do not display the dependence of $f(x;x_0)$ on $x_0$. The function $f(x)$ in \eqref{edo1} comes from condition (ii). We will prove the theorem by assuming that $f(x)={R}_x(x_0,x)$. The proof for the case $f(x)={R}_y(x,x_0)$ is completely analogous. 

Let $x(t)$ be the solution of \eqref{edo1}. We define the following  functions
$$
\phi(t):=x(t(\alpha+\beta))
\quad \mbox{and} \quad
\psi(t):={R}(x(t\alpha),x(t\beta)),
$$
which are of class $\mathcal{C}^k$, with $k\geq 1$ and $k\geq 2$, respectively.

We will show that $\phi(t)=\psi(t)$ for all $t\in I$, being $I$ the maximal interval where $x(t)$ is defined. As a consequence, we will obtain that $x(t\alpha+t\beta)={R}(x(t\alpha),x(t\beta))$, and therefore, by setting $\alpha=1$ and $\tau=t\beta$,
$$
x(t+\tau)={R}(x(t),x(\tau)).
$$

By using the above definition of $\phi(t)$, a straightforward computation shows that $\phi(t)$ is the unique solution of the following initial value problem
\begin{equation}
\label{eq.phi}
\dot{\phi}(t)=(\alpha+\beta){R}_x(x_0,\phi(t)),\quad \phi(0)=x(0)=x_0.
\end{equation}

Now, from the definition of $\psi(t)$ we get
\begin{eqnarray*}
\dot{\psi}(t)&=&{R}_x(x(t\alpha),x(t\beta))\,\dot{x}(t\alpha)\,\alpha+{R}_y(x(t\alpha),x(t\beta))\,\dot{x}(t\beta)\,\beta,\\
&=&{R}_x(x(t\alpha),x(t\beta))\,{R}_x(x_0,x(t\alpha))\,\alpha+{R}_y(x(t\alpha),x(t\beta))\,{R}_x(x_0,x(t\beta))\,\beta.
\end{eqnarray*}
Then, by applying condition (ii) we rewrite $\dot{\psi}(t)$ as
$$
\dot{\psi}(t)={R}_x(x(t\alpha),x(t\beta))\,{R}_x(x_0,x(t\alpha))\,\alpha+{R}_y(x(t\alpha),x(t\beta))\,{R}_y(x(t\beta),x_0)\,\beta.
$$
Hence, by using (iii), (ii) and the definition of $\psi(t)$ we get
$$
\dot{\psi}(t)= (\alpha+\beta)\,{R}_x(x_0,{R}(x(t\alpha),x(t\beta))) =(\alpha+\beta)\,{R}_x(x_0,\psi(t)).
$$
Additionally, condition (i) implies that $\psi(0)=R(x_0,x_0)=x_0$  
Therefore, we have proved that $\psi(t)$ satisfies the initial value problem
\begin{equation}
\label{eq.psi}
\dot{\psi}(t)=(\alpha+\beta)\,{R}_x(x_0,\psi(t)), 
\quad 
\psi(0)=x_0.
\end{equation}
From \eqref{eq.phi} and \eqref{eq.psi} is clear that $\phi(t)$ and $\psi(t)$ are the solutions of the same initial value problem. Hence we have $\psi(t)=\phi(t)$. This completes the proof
\end{proof}


\section{Double-angle Formulas and applications.}
\label{sec:Computation}
In this section we will show how the results of previous section can be used to get an algorithm, based on the existence of double-angle formulas, to obtain accurate approximations of solutions of initial value problems in autonomous first order differential equations.

Recall that a double-angle formula for a real function $\phi(t)$ is a function $R(x)$ satisfying the following relation 
$$
\phi(2t)=R(\phi(t)).
$$
 Under some assumptions the existence of a double-angle formula is established, as it showed by the next result, which is a straightforward consequence of Theorem~\ref{theo:AdditionTheorem}.
\begin{coro}
\label{coro:AdditionTheorem}
Let $I$ be a real open interval. If $\phi:I \longrightarrow \mathbb{R}$ is a $\mathcal{C}^k$ function with non empty addition domain $\mathcal{D}_{I}$ and with non vanishing derivative on $I$, then $\phi(t)$ is endowed with a  $\mathcal{C}^k$ double-angle formula.
\end{coro}

In what follows, we will denote both the explicit addition formula and the double-angle formula of a function $\phi(t)$ with the same symbol $R$, and the context will make clear which one are we referring to.

In general, the problem of finding the addition or double-angle formula of a given function is a difficult task. However, the previous results show the relationship between the direct and inverse problems of addition theorems for solutions of autonomous first order ordinary differential equations. More precisely, let $x(t)$ be the solution of the initial value problem 
\begin{equation}\label{eq:pvi}
\dot{x}=f(x),\quad x(0)=x_0.
\end{equation}
If $f(x_0)= 0$, then the constant function $x(t)\equiv x_0$ is the solution to the problem and trivially $R(x,y)=x_0$ is its explicit addition formula. If  $f(x_0)\neq 0$, then
Theorem~\ref{theo:AdditionTheorem} 
ensures the existence of $R(x,y)$ for $x(t)$.  Recall that the exponential or trigonometric functions are solutions of this kind of initial value problems, which have explicit addition formulas. 

On one hand, while the existence of $R(x,y)$ is guaranteed, knowing its concrete expression is not easy and only a few examples are known in the literature. On the other hand, as far as we know, there is no systematic method to compute the expression of $R(x,y)$ for a given function $\phi(t)$ or for the solution $x(t)$ of \eqref{eq:pvi}, with  $f(x)$ an arbitrary function. Nevertheless, this problem could be tackled for the case when $f(x)$ has a sort of simplicity. For instance, if $ n\in \mathbb{N}$, 
then for $f(x)=x^{n+1}$ and $f(x)=x^{1/(n+1)}$ we obtain that 
$$
R(x,y)=\dfrac{x_0 \,x \,y}{\left(x_0^n( x^n +y^n)-x^n y^n\right)^{1/n}}
$$
and
$$
R(x,y)=\left( x^{{n}/{(n+1)}}-x_0^{{n}/{(n+1)}}+y^{{n}/{(n+1)}}\right)^{{(n+1)}/{n}}
$$
are the concrete explicit addition formulas for the solution $x(t)$ of \eqref{eq:pvi}.

Analogously to the case of explicit addition formulas, the existence of a double-angle formula $R(x)$, for the solution $x(t)$ of \eqref{eq:pvi}, is guaranteed by Corollary \ref{coro:AdditionTheorem} and its concrete expression of could be difficult to obtain. Nevertheless, we point out that it is possible take advantage of the
existence of the double-angle formula $R(x)$ to derive a procedure to compute the Taylor polynomial of any degree of $R(x)$ at $x_0$. Such a Taylor  polynomial can be used to produce an algorithm that allow us to provide an approximation of $x(t)$. More precisely, these Taylor polynomials can be used to produce a succession of functions converging uniformly to $x(t)$ in a certain compact domain $D$, such a convergence and an iterative procedure of the Taylor polynomial of enough large degree allow us to give a accurate approximation to $x(t)$ in $D$, which is contained in the maximal interval of the solution $x(t)$.  This mentioned numerical scheme will be dubbed as duplication algorithm which we will explain in the follows.

\subsection{Taylor polynomials of $R(x)$}
Let $x(t)$ be the solution of \eqref{eq:pvi} and $R(x)$ its double-angle formula, that is,
\begin{equation}
\label{FormulaAnguloDoble}
x(2t)=R(x(t)).
\end{equation}
We will describe the procedure to compute the Taylor polynomial of $R(x)$. We consider the non trivial case, that is, $f(x_0)\neq 0$ and by simplicity, we assume that $f(x)$ in \eqref{eq:pvi} is analytic, and so are $x(t)$ and $R(x)$. To obtain the Taylor polynomial of $R(x)$, we compute the successive derivatives of $R(x)$ at $x_0$. 

Firstly, from \eqref{FormulaAnguloDoble} it is clear that 
\begin{equation}
\label{Taylor-Pol-ord-0}
R(x_0)=x_0.
\end{equation}
We now take the derivative of \eqref{FormulaAnguloDoble} with 
respect to $t$:
$$
2\dot{x}(2t)=R'(x(t))\dot{x}(t),
$$
here $R'(\cdot)$ means the derivative of $R(x)$ with respect to $x$. Then, by using \eqref{eq:pvi} in the left- and right-hand sides we get
\begin{equation}
\label{Derivada-ord-1}
2f(x(2t))=R'(x(t))f(x(t)).
\end{equation}
Then, by setting $t=0$ we obtain $2f(x_0)=R'(x_0)f(x_0),$ which, taking into account that $f(x_0)\neq0$, yields
\begin{equation}
\label{Taylor-Pol-ord-1}
R'(x_0)=2.
\end{equation}
For obtaining the second order derivative of $R(x)$, we rewrite equation \eqref{Derivada-ord-1} by using \eqref{FormulaAnguloDoble}, which leads to
$$
2f(R(x(t)))=R'(x(t))f(x(t)).
$$
Then, the derivative of this last expression with respect to $t$ is
$$
2f'(R(x(t)))R'(x(t))\dot{x}(t)=R''(x(t))\dot{x}(t)f(x(t))+R'(x(t))f'(x(t))\dot{x}(t),
$$
which, by using \eqref{eq:pvi} and by setting $t=0$, becomes
\begin{eqnarray}
2f'(x_0)R'(x_0)\,f(x_0)=R''(x_0)(f(x_0))^2+R'(x_0)f'(x_0)f(x_0).\nonumber
\end{eqnarray}
Finally, by simplifying terms and by using \eqref{Taylor-Pol-ord-1} we arrive to
\begin{equation}
\label{Taylor-Pol-ord-2}
R''(x_0)=2 f'(x_0)/f(x_0).
\end{equation}
The third order derivative $R'''(x_0)$ is easily calculated by executing the same procedure
$$
R'''(x_0)=6f''(x_0)/f(x_0).
$$
Hence, by using this last equation and equations \eqref{Taylor-Pol-ord-0}, \eqref{Taylor-Pol-ord-1}, and  \eqref{Taylor-Pol-ord-2}, the 3th order Taylor polynomial of   $R(x)$ at $x_0$ is
$$
x_0+2(x-x_0)+\frac{f'(x_0)}{f(x_0)}(x-x_0)^2+\frac{f''(x_0)}{f(x_0)}(x-x_0)^3.
$$

Following the same methodology, higher order derivatives of $R(x)$ at $x=x_0$ can be computed. Thus, allowing to obtain the Taylor polynomials of arbitrary order for  $f(x)$ analytic and up to order $n\leq k$ for $f(x)$ of class $\mathcal{C}^k$. We include the expression of the derivatives up to the 10th order in the Appendix~\ref{sec:Appendix}.


\subsection{The duplication algorithm}
\label{Algorithm}
In this section, we use double-angle formulas to provide an estimation for the solution $x(t)$ of an initial value problem. 

Our method is based on the use of the double-angle formula $R(x)$ to generate local approximations for the function $x(t)$ in a compact subset $D$ containing $t=0$ and contained in the maximal interval $x(t)$.  The following procedure will be dubbed as the \emph{duplication algorithm}. Specifically, we compute $x(t)$ in a small neighborhood $V_0$ of the origin; for $\hat{t}\in D$ far from the origin, we divide $\hat{t}$ by two as many times as it is needed to get $\hat{t}/2^n\in V_0$, and we compute $x(\hat{t}/2^n)$; then, we can recover  $x(\hat{t})$  by iterating $n$-times the double-angle formula. That is,
\begin{equation}
\label{eq:DuplicationAlgorithm}
x(\hat{t}\,)= R^n\left(x\left(\frac{\hat{t}}{2^n}\right)\right),
\end{equation}
where $R^n$ denotes the $n$-th iteration of  the double-angle formula $R$ associated to $x(t)$. 

The previous process depends on the availability of the double-angle formula $R$, as well as the capacity of computing $x(t)$ in $V_0$. However, when $x(t)$ is unknown, an approximation may be obtained by using a Taylor polynomial 
\begin{equation}
\label{eq:DuplicationAlgorithm2}
x(\hat{t}\,)\approx R^n\left(x_{m_1}\left(\frac{\hat{t}}{2^n}\right)\right),
\end{equation}
being $x_{m_1}(t)$ the $m_1$-th order Taylor polynomial of $x(t)$ at $t=0$. Moreover, for the case in which we are not able to find the exact double-angle formula $R$, we will use its Taylor polynomials. For this purpose, the previous section gave an easy method to compute arbitrary derivatives of $R$ at $x=x_0$. This situation leads us to the following local approximation
\begin{equation}
\label{eq:DuplicationAlgorithm2}
x(\hat{t}\,)\approx R^n_{m_2}\left(x_{m_1}\left(\frac{\hat{t}}{2^n}\right)\right),
\end{equation}
where $R^n_{m_2}$ is the $n$-th iteration of the $m_2$-th order Taylor polynomial of $R$.  In fact, straightforward computations prove that
\begin{equation}
\label{eq:DuplicationAlgorithm3}
x(\hat{t}\,)= R^n_{m_2}\left(x_{m_1}\left(\frac{\hat{t}}{2^n}\right)\right)+O\left(\left(\frac{\hat{t}}{2^n}\right)^{\!\!m}\right),\quad \mbox{with  $m=\min \{m_1+1,m_2+1\}$.}
\end{equation}

Keeping in mind the assumption that $f(x)$ in \eqref{eq:pvi} is analytic, we have that $x(t)$ is analytic in its maximal domain ${I}_{M}=(\alpha,\beta)$. Thus, the associated double-angle formula $R(x)$ is analytic in $\mathcal{D}_{{I}_{M}}=x(\alpha/2,\beta/2)$.  Now we construct a polygonal function defined in a compact interval $D=[a,b]\subset I_M$ to estimate $x(t)$, which is based on local approximations and their interpolation. More precisely, given the $r/2^n$-sized partition $\mathcal{P}(D)=\{ t_0,t_1,\ldots,t_{2^n} \}$, with $t_0=a$, $t_{2^n}=b$ and $r=b-a$ we define for $t\in\left[t_i,t_{i+1}\right]$
\begin{equation}
\label{eq:DefPMN}
P_n^{m_1,m_2}(t)=R_{m_2}^n\left(x_{m_1}({t_i}/{2^n})\right)\dfrac{t-t_{i+1}}{t_i-t_{i+1}}+R_{m_2}^n\left(x_{m_1}({t_{i+1}}/{2^n})\right)\dfrac{t_{i}-t}{t_i-t_{i+1}}.
\end{equation}

\begin{remark}
The duplication algorithm uses a Taylor approximation of $x(t)$ in a small neighborhood of the origin 
$V_0=[-r_0,r_0].$
Note that the partition $\mathcal{P}(D)$ can be always be chosen in such a way that $t_i/2^n\in V_0,$ for any $t_i\in \mathcal{P}(D).$ In what follows we will assume that this condition is satisfied.
\end{remark}

\begin{remark}
The above function $R^n_{m_2}$ admits an alternative definition as the $m_2$-th order Taylor polynomial of the function $R^n$, which can be obtained by combining the successive derivatives given in Appendix~\ref{sec:Appendix} and the formula of Fa\`a di Bruno \cite{Krantz2002}. 
 In that case, the numerical results could be improved. However,  this subject is left for further study.
\end{remark}

In the proof of the next theorem we need the following auxiliary polygonal function
\begin{equation}
\label{eq:DefPMN2}
P_n(t)=x(t_i)\dfrac{t-t_{i+1}}{t_i-t_{i+1}}+x(t_{i+1})\dfrac{t_{i}-t}{t_i-t_{i+1}},\quad \mbox{ for } t\in\left[t_i,t_{i+1}\right].
\end{equation}

\begin{theorem}[Duplication Algorithm Convergence]
The family $\{P_n^{m_1,m_2}\}_{m_1,m_2,n \in \mathbb{N}}$ of polygonal functions converges uniformly to $x(t)$ in any compact interval $D\subset I_M$ as $n\rightarrow\infty$. 
Moreover, we have the following error estimation
$$
\big{|}P_n^{m_1,m_2}(t)-x(t)\big{|}\leq O\left[(r/2^n)^{2}\right] \qquad \forall\, t\in D.
$$
\end{theorem}
\begin{proof}
Let us consider  $t\in\left[t_i,t_{i+1}\right]\subset D$. Then by the triangle inequality we have
\begin{equation}
\label{eq:Descomposicion}
\big{|}P_n^{m_1,m_2}(t)-x(t)\big{|}\leq\big{|}P_n^{m_1,m_2}(t)-P_n(t)\big{|}+\big{|}P_n(t)-x(t)\big{|}.
\end{equation}
where $P_n(t)$ is defined in \eqref{eq:DefPMN2}. By using \eqref{eq:DefPMN}, we get 
$$
\begin{aligned}
\label{eq:PrimerSumando}
\big{|}P_n^{m_1,m_2}(t)-P_n(t)\big{|}=&\left|\left(R^n_{m_2}(x_{m_1}(t_i/2^n))-x(t_i))\right)\dfrac{t-t_{i+1}}{t_i-t_{i+1}}\right.\\
&\left.+\left(R^n_{m_2}(x_{m_1}(t_{i+1}/2^n))-x(t_{i+1}))\right)\dfrac{t_{i}-t}{t_i-t_{i+1}}\right|.
\end{aligned}
$$
Hence, from \eqref{eq:DuplicationAlgorithm3} it follows that
$$
\begin{aligned}
\label{eq:PrimerSumando4}
\big{|}P_n^{m_1,m_2}(t)-P_n(t)\big{|}=&\left|O\left[(t_i/2^n)^{m}\right]\dfrac{t-t_{i+1}}{t_i-t_{i+1}}+O\left[(t_{i+1}/2^n)^{m}\right]\dfrac{t_{i}-t}{t_i-t_{i+1}}\right|
\end{aligned}
$$
and we obtain the following upper bound by considering that $|t_j|\leq r$ and that $m\geq 2$,
$$
\begin{aligned}
\label{eq:PrimerSumando5}
\big{|}P_n^{m_1,m_2}(t)-P_n(t)\big{|}\leq&\left|O\left[(t_i/2^n)^{m}\right]\dfrac{t-t_{i+1}}{t_i-t_{i+1}}\right|+\left|O\left[(t_{i+1}/2^n)^{m}\right]\dfrac{t_{i}-t}{t_i-t_{i+1}}\right|\\
\leq&\left|O\left[(r/2^n)^{m}\right]\dfrac{t-t_{i+1}}{t_i-t_{i+1}}\right|+\left|O\left[(r/2^n)^{m}\right]\dfrac{t_{i}-t}{t_i-t_{i+1}}\right|\\
\leq&\,O\left[(r/2^n)^{m}\right] \leq\,O\left[(r/2^n)^{2}\right].
\end{aligned}
$$

To finish the proof, we look for an upper bound of the second term in \eqref{eq:Descomposicion}
$$
\begin{aligned}
\big{|}P_n(t)-x(t)\big{|}=&\left|x(t_i)\dfrac{t-t_{i+1}}{t_i-t_{i+1}}+x(t_{i+1})\dfrac{t_{i}-t}{t_i-t_{i+1}}-x(t)\right|\\
=&\left|t\dfrac{x(t_i)-x(t_{i+1})}{t_i-t_{i+1}}+\dfrac{t_i}{t_i-t_{i+1}}x(t_{i+1}) -\dfrac{t_{i+1}}{t_i-t_{i+1}}x(t_{i})-x(t)\right|\\
=&\left|t\dfrac{\Delta x_i}{\Delta t_i}+\dfrac{t_i}{\Delta t_i}x(t_{i+1}) -\dfrac{t_{i+1}}{\Delta t_i}x(t_{i})-x(t)\right|,
\end{aligned}
$$
where $\Delta x_i:=x(t_i)-x(t_{i+1})$ and $\Delta t_i:=t_i-t_{i+1}$. 
If now add and subtract $\dfrac{t_{i+1}}{\Delta t_i}x(t_{i+1})$, then we get
$$
\begin{aligned}
\big{|}P_n(t)-x(t)\big{|}=&\left|t\dfrac{\Delta x_i}{\Delta t_i}+\dfrac{t_i-t_{i+1}}{\Delta t_i}x(t_{i+1}) +\dfrac{t_{i+1}}{\Delta t_i}(x(t_{i+1})-x(t_i))-x(t)\right|\\
=&\left|(t-t_{i+1})\dfrac{\Delta x_i}{\Delta t_i}+x(t_{i+1})-x(t)\right|\\
=&\left|\dfrac{\Delta x_i}{\Delta t_i}+\dfrac{x(t_{i+1})-x(t)}{t-t_{i+1}}\right|(t_{i+1}-t)\\
=&\,O\left[(r/2^n)^{2}\right].
\end{aligned}
$$
The last equality follows from the fact that 
both factors in the previous step satisfy that 
$$\left|\dfrac{\Delta x_i}{\Delta t_i}+\dfrac{x(t_{i+1})-x(t)}{t-t_{i+1}}\right|=O\left[(r/2^n)\right],\quad (t_{i+1}-t)=O\left[(r/2^n)\right].$$
The proof has been completed.
\end{proof}

In previous proof is stated that 
$$\big{|}P_n^{m_1,m_2}(t)-x(t)\big{|}\leq\big{|}P_n^{m_1,m_2}(t)-P_n(t)\big{|}+\big{|}P_n(t)-x(t)\big{|}.
$$
We have showed that $\big{|}P_n^{m_1,m_2}(t)-P_n(t)\big{|}$ depends on the minimum of the orders of the Taylor polynomials involved and $\big{|}P_n(t)-x(t)\big{|}$ depends on the size of the partition $\mathcal{P}(D)$. The numerical experiments carried out in next section will show such dependence.

\section{Numerical Simulations}
\label{sec:Simulations}
In this section we analyze the efficiency and accuracy of the duplication algorithm as an approximation method for the solution of ordinary differential equations. We do so by comparing it with standard numerical methods. More precisely, our experiments have been carried out by using the software \emph{Wolfram Mathematica}, version 12.1.1.0 Linux x86 (64-bit). This software is running on the platform {\color{black} Intel® Core™ i3-9100 CPU, 3.60GHz$\times$4}.

In all our simulations we consider the following notation:
\begin{itemize}
\setlength{\itemsep}{0.15cm}
\item $x(t)$ is the exact solution of the initial value problem \eqref{eq:pvi}. 
\item $\hat{x}(t)$ is the approximate solution to \eqref{eq:pvi} given by the duplication algorithm having the exact double-angle formula. 
\item $\tilde{x}(t)$ is the approximate solution to \eqref{eq:pvi} by using the duplication algorithm with the Taylor approximation of the double-angle formula.
\item $\check{x}(t)$ is the numeric approximation to \eqref{eq:pvi} provided by {\bf \tt NDSolve} of 
\emph{Mathematica}.
\end{itemize}

\subsection{First Example.}
\label{sec:ToyModel}
We consider
$$
\dot{x}=x^2,\:\: x(0)=x_0.
$$
The exact solution $x(t)$ of this problem and its associated double-angle formula $R(x)$ can be easily obtained after some algebraic manipulations. They have following expressions
$$
x(t)=\frac{x_0}{1-x_0t}\quad \mbox{and} \quad R(x)=\frac{x_0 x}{2x_0 -x}.
$$

In order to test the accuracy and efficiency of the duplication algorithm in this toy example, we will compute the errors $\hat{x}(t)-{x}(t)$, $\tilde{x}(t)-{x}(t)$ and $\check{x}(t)-{x}(t)$ as well as the computer time  of these errors on a compact interval contained in the domain of $x(t)$. To perform the numerical simulations we will consider $x_0=1$, thus $x(t)$ is defined in $(-\infty, 1)$. It is easy to check that $t=1$ is an asymptote for $x(t)$. With the aim of studying the behaviour of the duplication algorithm in a generic interval and for $t$ near the asymptote, we consider simulations in the intervals $I_1=\left[-0.5 ,0.5 \right]$ and $I_2=\left[-0.5 ,0.99 \right]$.

The approximation $\hat{x}(t)$ has been obtained by partitioning the integration intervals $I_1$ and $I_2$ in 240 and 10.000 points respectively. As we can see in Figure \ref{diagrama1}, the difference between the exact solution and the approximation by the duplication algorithm, with the exact double-angle formula $R(x)$, produces an error of order of magnitude $10^{-9}$ with a computer time of $0.0019$ seconds in $I_1$. Note that, as we approach the asymptote $t=1$ in the interval $I_2$, the accuracy of the simulation get worse and the computational effort increases.
\begin{figure}[h]
\includegraphics[scale=0.37]{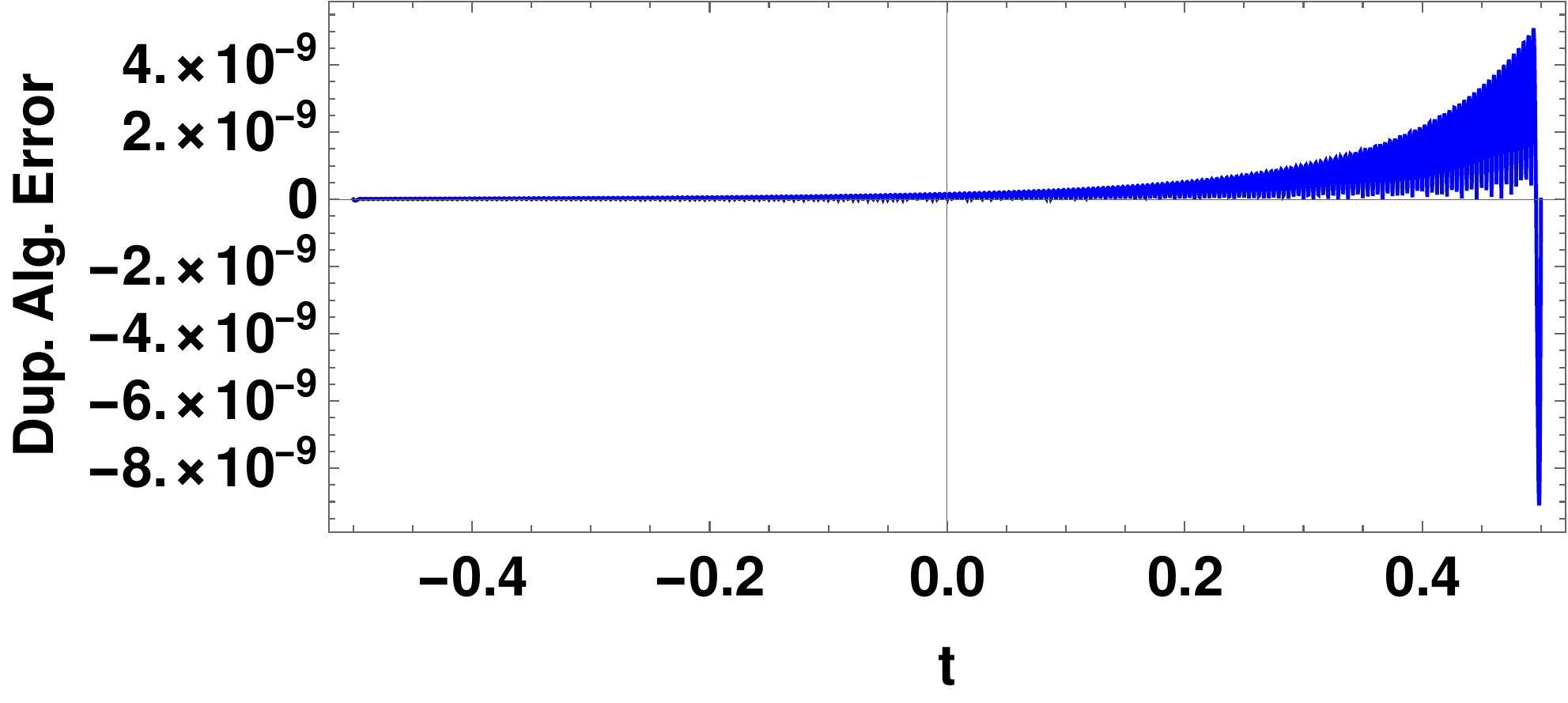}\,
\includegraphics[scale=0.37]{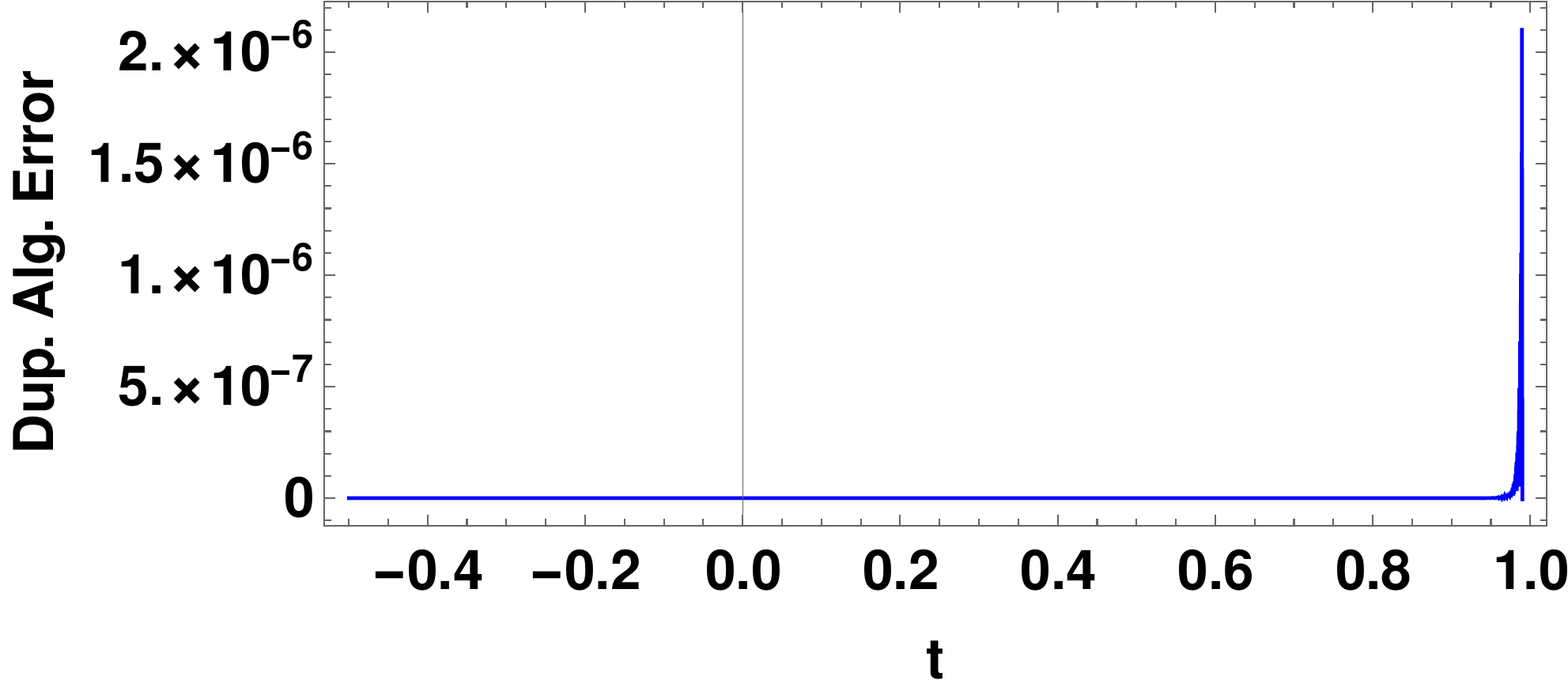}
\caption{\label{diagrama1} Propagation of $\hat{x}(t)-{x}(t)$ in $I_1$ with computer time of $0.0019$ s (left) and $I_2$ with computer time of $0.2845$ s (right).}  
\end{figure}

In Figure \ref{diagrama2} we show the error propagation between the exact solution $x(t)$ and $\tilde{x}(t)$, which we obtain by numerical simulation applying the duplication algorithm with an approximation of $R(x)$ given by a 20-th order Taylor polynomial. The error in this simulation behaves as in the case of $\hat{x}(t)$ with almost identical computer times.
\begin{figure}[h]
\includegraphics[scale=0.37]{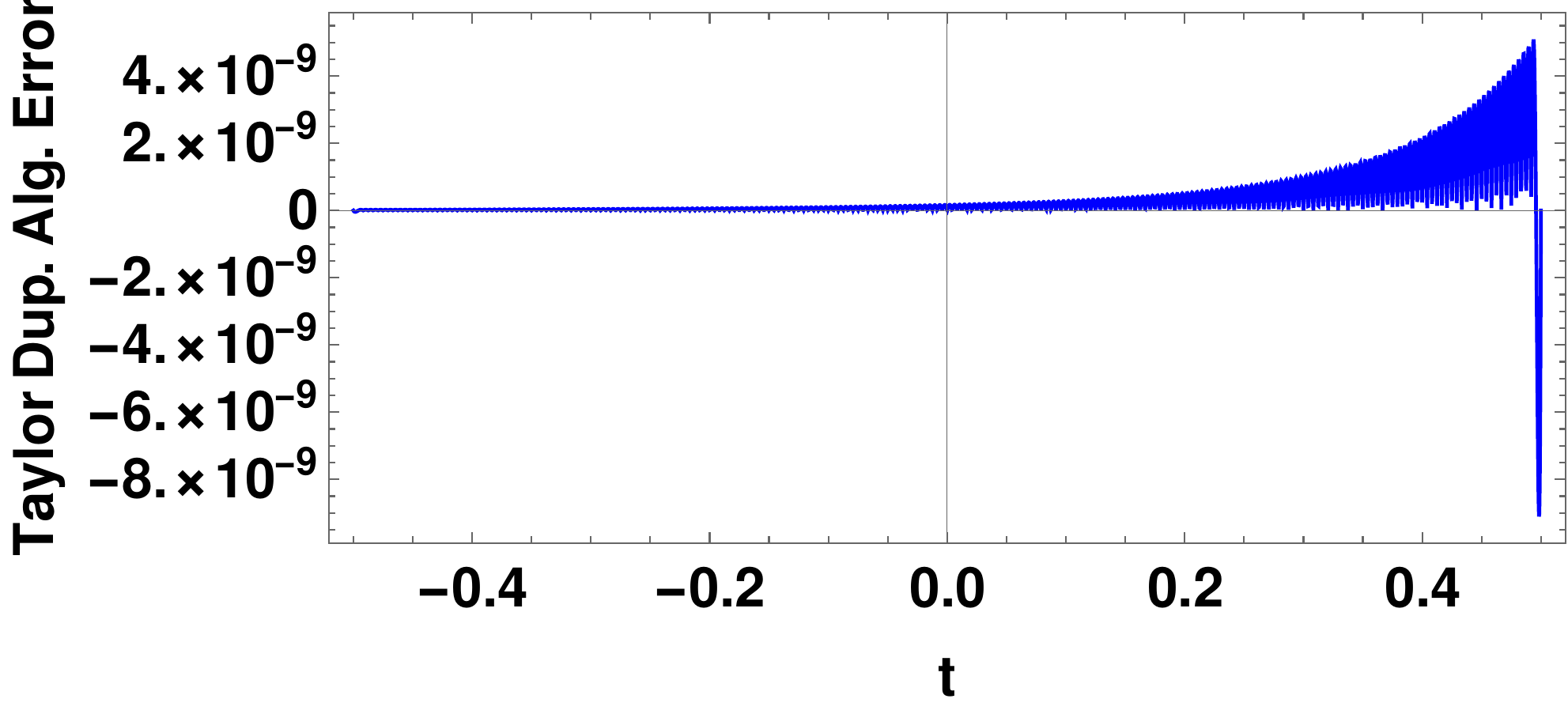}\,
\includegraphics[scale=0.37]{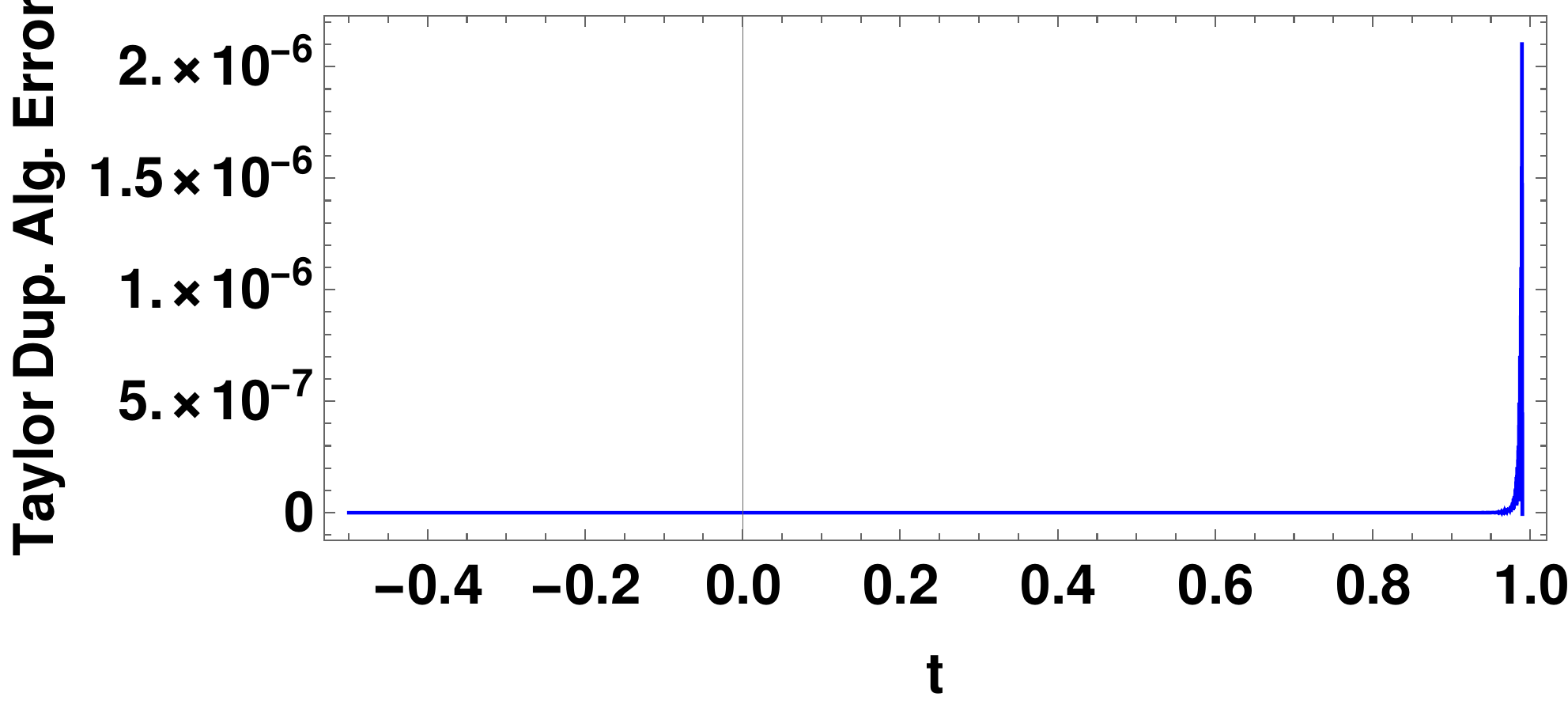}
\caption{\label{diagrama2}Propagation of $\tilde{x}(t)-{x}(t)$ in $I_1$ (left) and $I_2$ (right) with a 20th Taylor polynomial for $R(x)$. Computer time $0.0022$ and $0.2838$ s for $I_1$ and $I_2$, respectively.}
\end{figure}

Remarkably, the order of magnitude of the error in the previous procedures is less than the one provided by \emph{Wolfram Mathematica}. Indeed, Figure \ref{diagrama3} shows the difference of the exact solution versus the standard numerical integration method {\bf \tt  NDSolve} \emph{Mathematica} commad with options: {Method = Automatic, AccuracyGoal= 20 and MachinePrecision} which has error of order of magnitude $10^{-8}$ and $10^{-3}$ in $I_1$ and $I_2$ respectively, with computer times  $0.0024$ and $0.0020$ seconds.
\begin{figure}[h]
\includegraphics[scale=0.37]{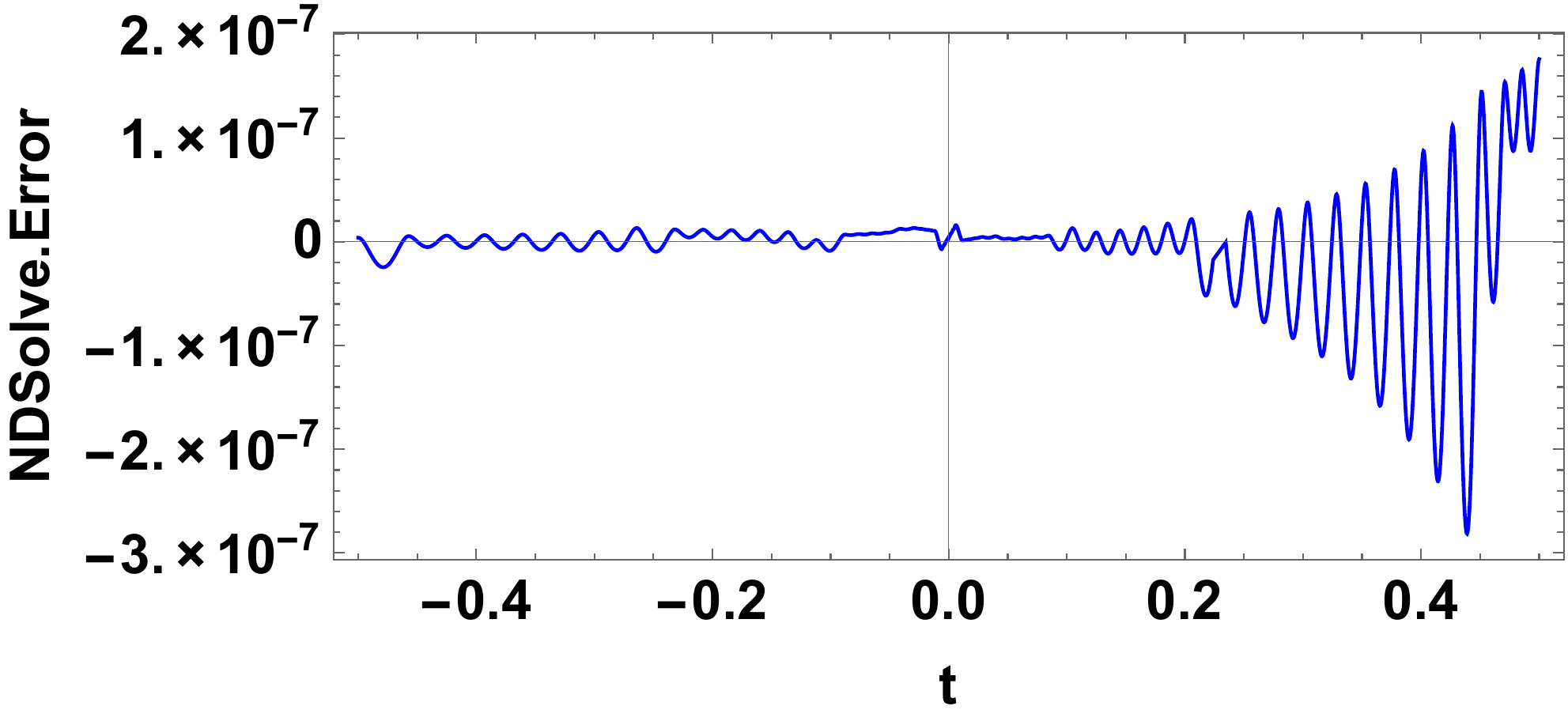}\,
\includegraphics[scale=0.37]{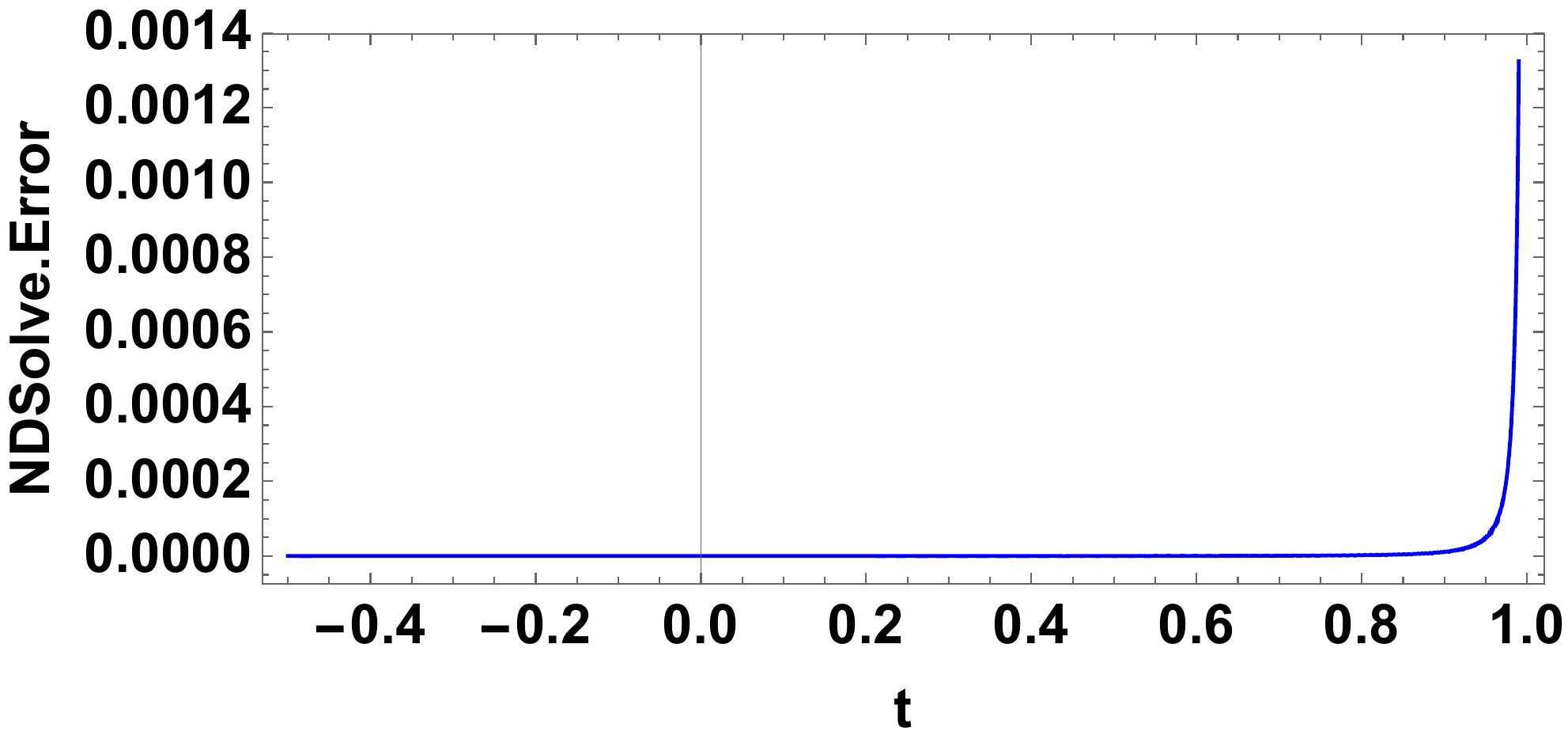}
\caption{\label{diagrama3}Propagation of $\check{x}(t)-{x}(t)$. Computer time of $0.0024$ and $0.0020$ s for $I_1$ and $I_2$, respectively. }
\end{figure}

In the following experiments we analyze the relation between the number of interpolation point and the accuracy of the approximation. To this aim we consider again $x_0=1$, in Figure~\ref{comparacion2} we simulate $x(t)$ in the interval $[-0.5,0.5]$ using 20th order  Taylor polynomial for the double-angle formula. As we increase the number of points in the interpolation process, computation time and accuracy increase linearly. However, for more that 4000 interpolation points, the error does not get better than $10^{-14}$.
\begin{figure}[h]
\label{diagrama5}
\includegraphics[scale=0.36]{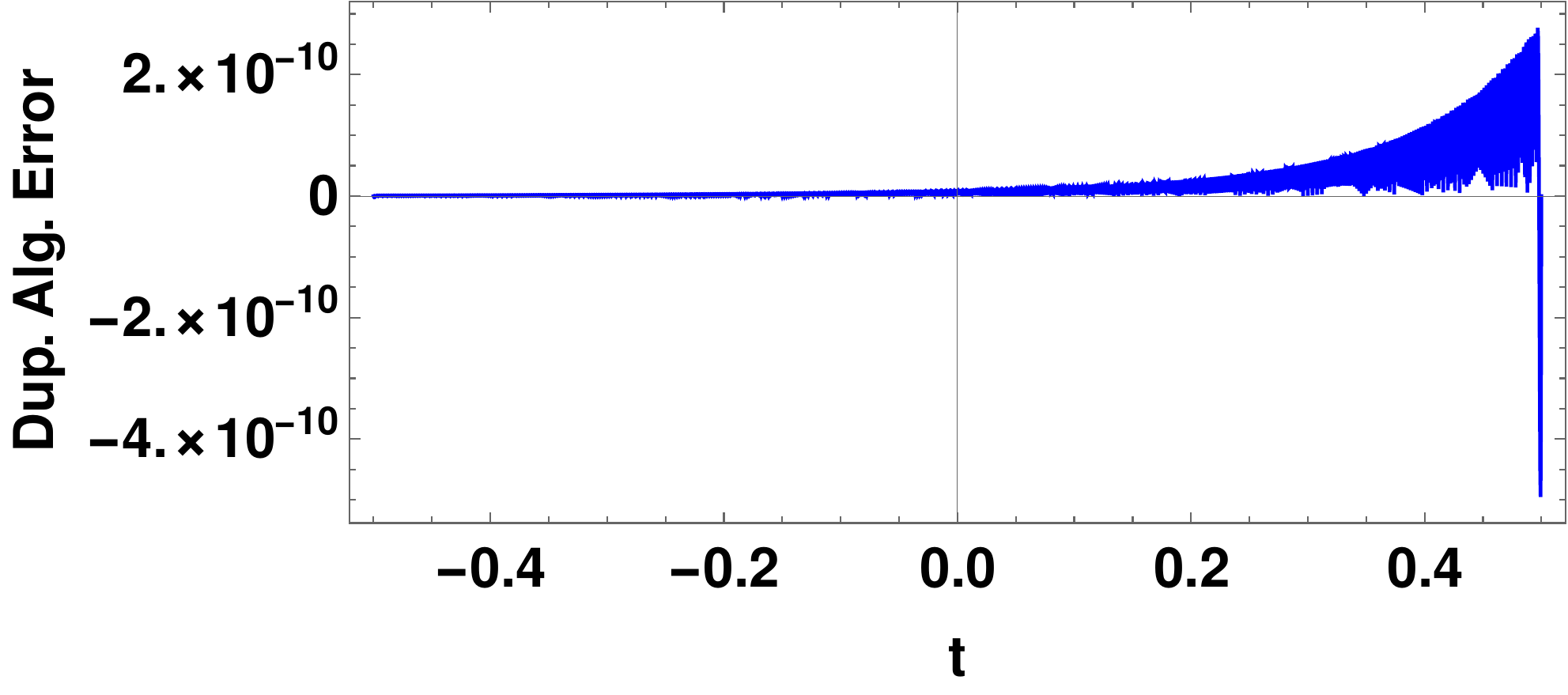}\quad
\includegraphics[scale=0.36]{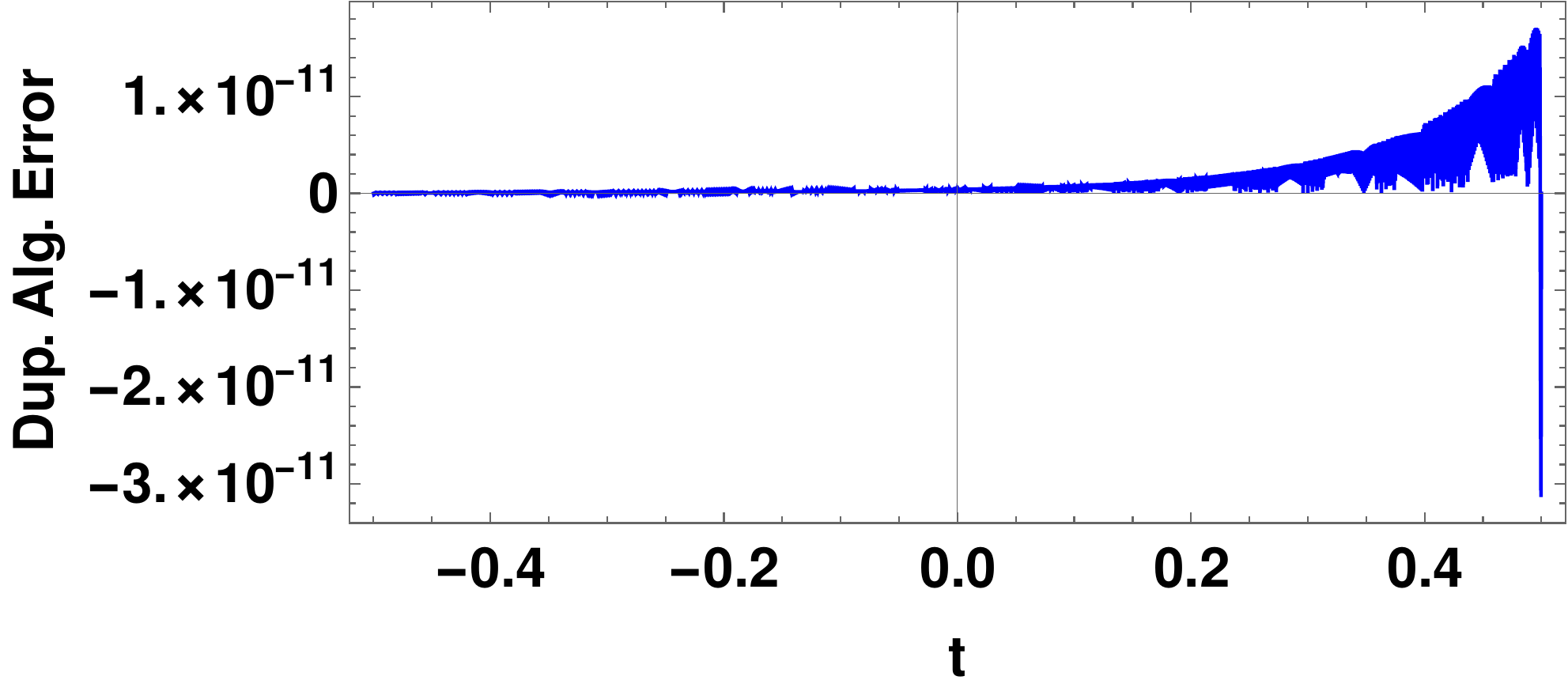}\\
\includegraphics[scale=0.36]{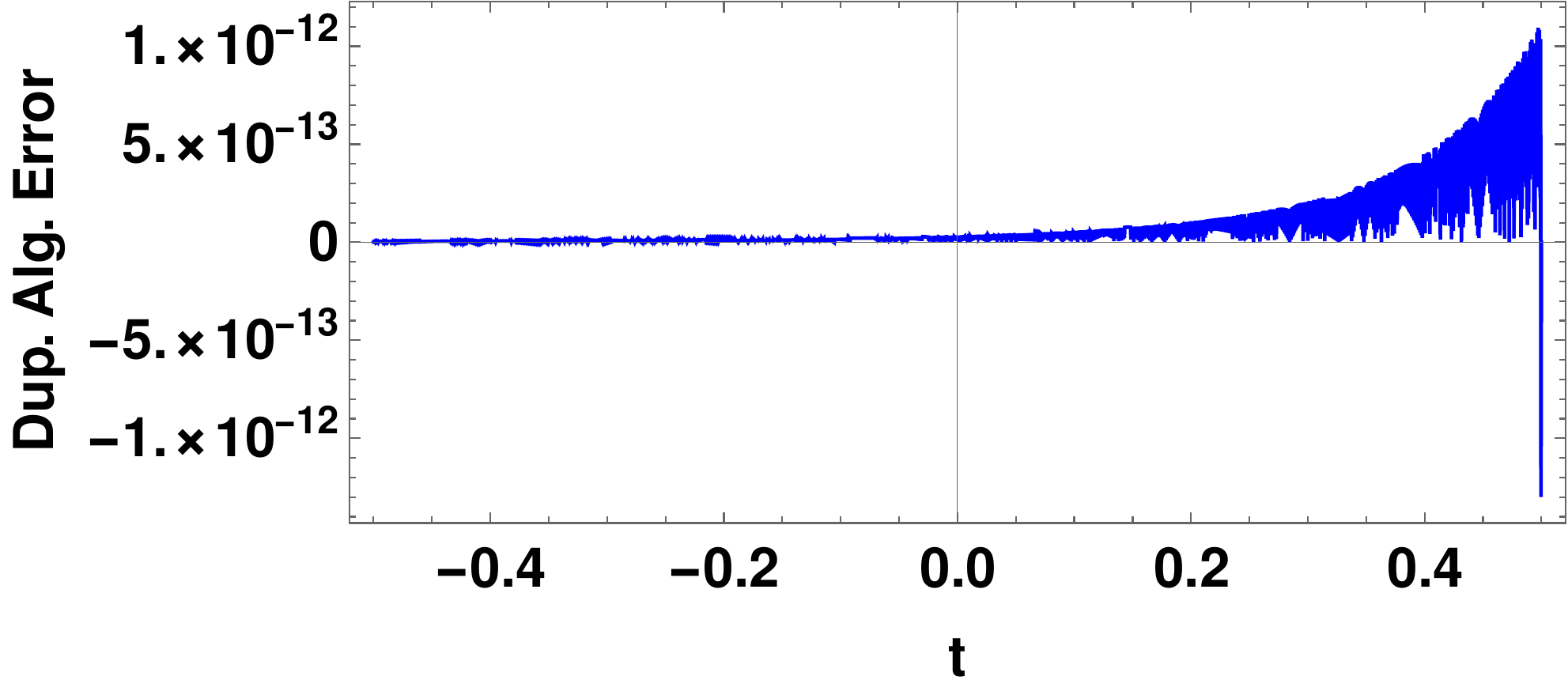}\quad
\includegraphics[scale=0.36]{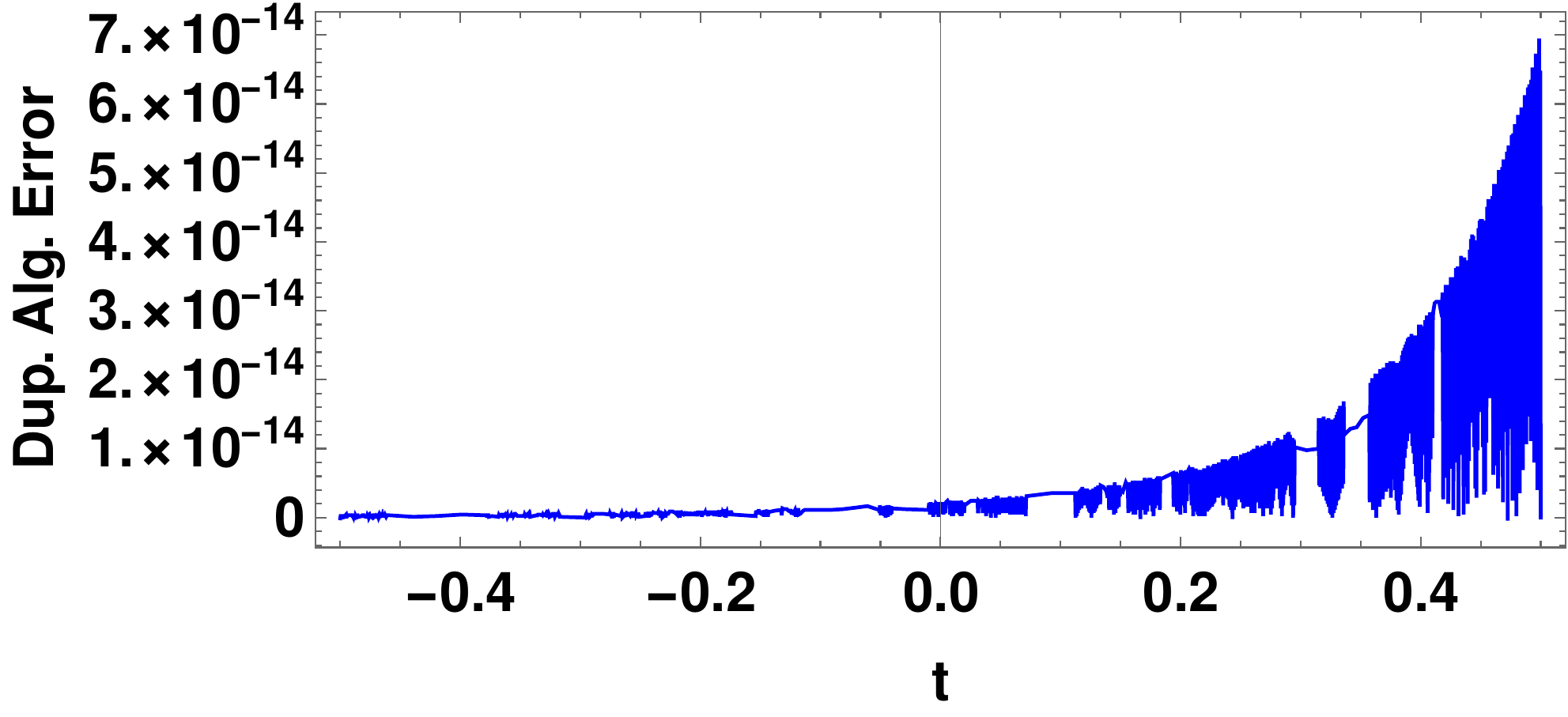}
\caption{Propagation of $\hat{x}(t)-{x}(t)$. Computer time equal to $0.0229$, $0.0400$, $0.0809$  and $0.1727$ s for  $500$, $1000$, $2000$ and $4000$ interpolation points, respectively.}
\label{comparacion2}
\end{figure}

\subsection{Second Example.}
\label{sec:SecondModel}
In the previous example we dealt with a simple equation and duplication algorithm. This was done with the aim of getting all the elements needed for the comparisons in a straightforward manner. Our next experiment addresses a more elaborated model, whose initial condition is close to the limit of the theory applicability. Precisely, the  second example is described by the following initial value problem
\begin{equation} 
\label{eq:Weierstrass}
\dot{x}=\sqrt{4 x^3-\frac{13 }{12}x-\frac{35}{216}},\quad x(0)=x_0=-\frac{5}{12}+\epsilon,
\end{equation}
where $|\epsilon| \ll 1$. 
As the reader may notice, the case $\epsilon=0$ implies that $\dot x(0)=0$. Thus, $x(t)$ does not fulfill the hypothesis of Corollary~\ref{coro:AdditionTheorem}, neither the initial value problem \eqref{eq:Weierstrass} satisfies the theorem of existence and uniqueness for ordinary differential equations. In this section we will consider an initial condition as close as possible to $x_0=-{5}/{12}$ to check if the performance of our method is affected by the limits of its applicability.

Equation \eqref{eq:Weierstrass} is written in the same way as the derivative of the $\wp$-Weierstrass elliptic function. However, the initial condition is evaluated in $t=0$ preventing $x(t)$ to have a pole in the origin as it is the case for $\wp$. Thus, we can consider the solution of the above equation as the restriction to the real domain of a $\wp$-Weierstrass elliptic function for which the lattice has been displaced from the origin in such a way that there is no poles in the real line. Alternatively, the explicit solution can be obtained in terms of the Jacobi elliptic functions \cite{Lawden1989}. Indeed, we have the following expression for the unique solution of \eqref{eq:Weierstrass} 
\begin{equation} 
\label{eq:WeierstrassSolution}
x(t)=\frac{7}{12}-\text{dn}\left(t+\delta_\epsilon \left|\frac{1}{4}\right.\right)^2,
\end{equation}
where $\delta_\epsilon$ and $\epsilon$ are related by the following formula
$$\text{dn}\left(\delta_\epsilon \left|{1}/{4}\right.\right)^2=1-\epsilon.$$

By using the double-angle formula of the Jacobi \text{dn}-function, and after some cumbersome algebraic manipulations, we get the exact double-angle formula for \eqref{eq:WeierstrassSolution}
\begin{eqnarray}
\label{eq:FormulaDobleEpsilon} 
R(x)=\frac{7}{12}-\left(\frac{ \alpha\,{P_2}\, ({P_2}-6 {P_3})}{{P_2}^2-96\,\epsilon\, {P_1}^2  }+\frac{ \beta \,{P_1} }{{P_2}^2-96 \,\epsilon\,{P_1}^2  }\right)^2,
\end{eqnarray}
where $\alpha=\sqrt{1-\epsilon }$, $\beta=2 \, \sqrt{6\,\epsilon-24 \,\epsilon^2 }$ and
$$
 {P_1}= \sqrt{(864 x^3-234 x-35)\left(144 x^2-168 x-59\right)^2},
\quad{P_2}=144 x^2+120 x-11,\quad{P_3}= 12 x-7.
$$
Therefore, as in the previous example, we can make comparisons between the exact solution $x(t)$ of \eqref{eq:Weierstrass}, the standard numerical approximation provided by \emph{Wolfram Mathematica} $\check{x}(t)$ and our duplication algorithm. 

\begin{figure}[h!]
\subfigure[Propagation of $\hat{x}(t)-{x}(t)$. Computer time equal to $0.0044$ s for 100 interpolation points. Taylor of order 20 for $x(t)$.]
{\label{MathematicaAprox}
\includegraphics[scale=0.35]{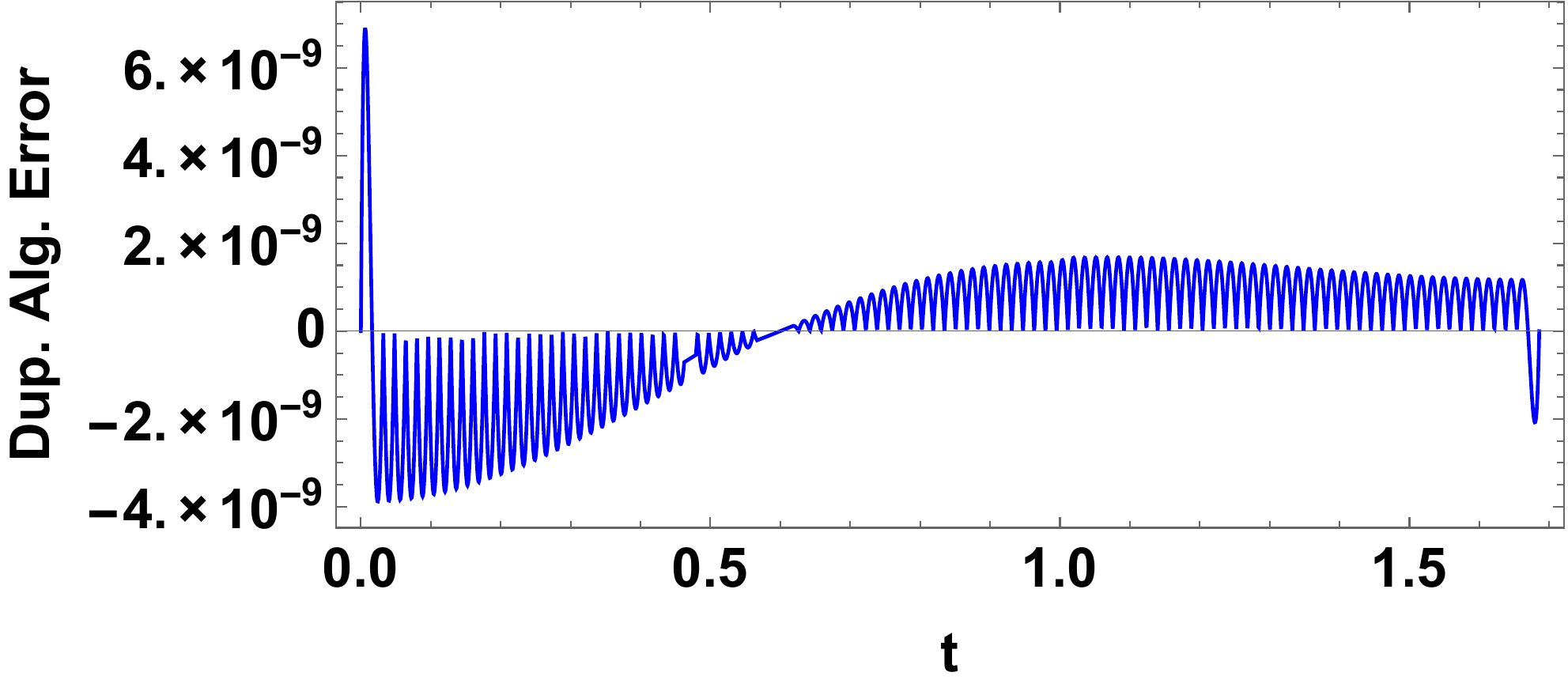}}\quad
\subfigure[Propagation of $\check{x}(t)-{x}(t)$. Computer time $0.0063$ s. {\bf \tt NDSolve} with options: Method = Automatic, AccuracyGoal= 20 and MachinePrecision.]
{\label{DuplicationAprox}
\includegraphics[scale=0.35]{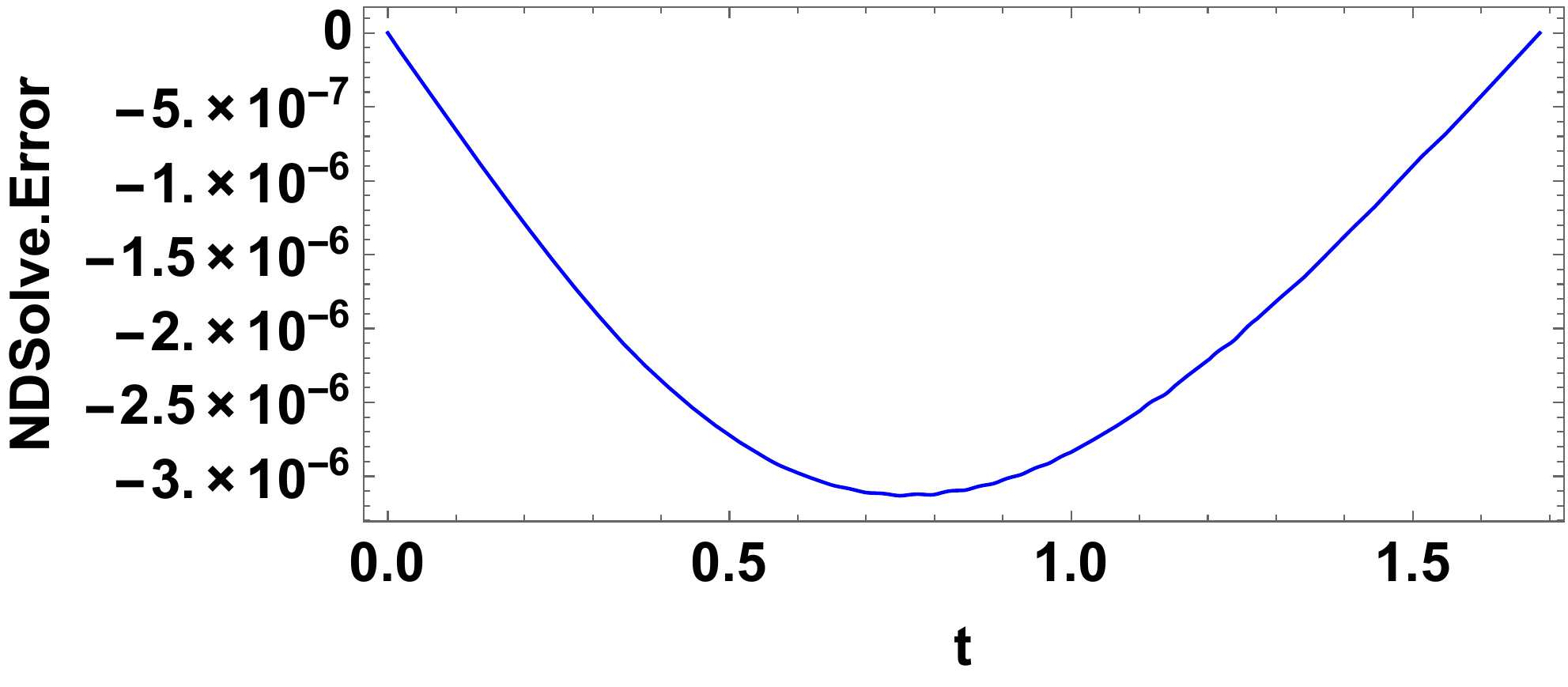}}
\caption{Comparison between the duplication algorithm and the {\bf \tt NDSolve} integrator for equation \eqref{eq:Weierstrass}.  The integration interval correspond to half a period of the function $x(t)$. Small parameters are set to $(\epsilon,\delta_\epsilon)= (2.5\times 10^{-31},10^{-15})$.}
\label{fig:ComparacionMathematicaDupAlgWeierstrassExacta}
\end{figure}

In Figure~\ref{fig:ComparacionMathematicaDupAlgWeierstrassExacta} we use the exact double-angle formula \eqref{eq:FormulaDobleEpsilon} to simulate the solution of \eqref{eq:Weierstrass} obtained with the duplication algorithm, then we compare it with the approximation given by the software package \emph{Wolfram Mathematica}. We observe that the accuracy of the duplication algorithm is two orders of magnitude better. Moreover, the computer time is on the same order for both approaches, with a slight improvement in the duplication algorithm. 

We also investigate how the order of the Taylor polynomials for the double-angle formula improves the approximations. In Figure~\ref{fig:ComparacionDupAlgWeierstrassOrdenTaylor}, we show the numerical experiments for orders 10, 15, 20 and 30. Setting the number of interpolation points to 100, we obtain that the error remains fixed around $10^{-10}$ for Taylor orders above 30.

Finally, we assess the role of the number of interpolation points in the accuracy of the approximation that we provide. More precisely, in  Figure~\ref{fig:ComparacionMathematicaDupAlgWeierstrassAprox} we use a Taylor polynomials of order 30th to approximate the double-angle formula. Then we show the evolution of the error when comparing with the exact solution for 80, 160, 320, 640, 1280 and 2560 interpolation points. As we duplicate the number of points, the error magnitude is approximately multiplied by $10^{-1}$. However, our experiments shows that, for this fixed Taylor approximation of the double-angle formula, the precision  stabilizes around $10^{-14}$ as we increase the number of interpolation points.

\begin{figure}[h!]
\subfigure[Propagation of $\tilde{x}(t)-{x}(t)$. Computer time equal to $0.0138$ s. Taylor of order 10.]
{\label{MathematicaAprox}
\includegraphics[scale=0.35]{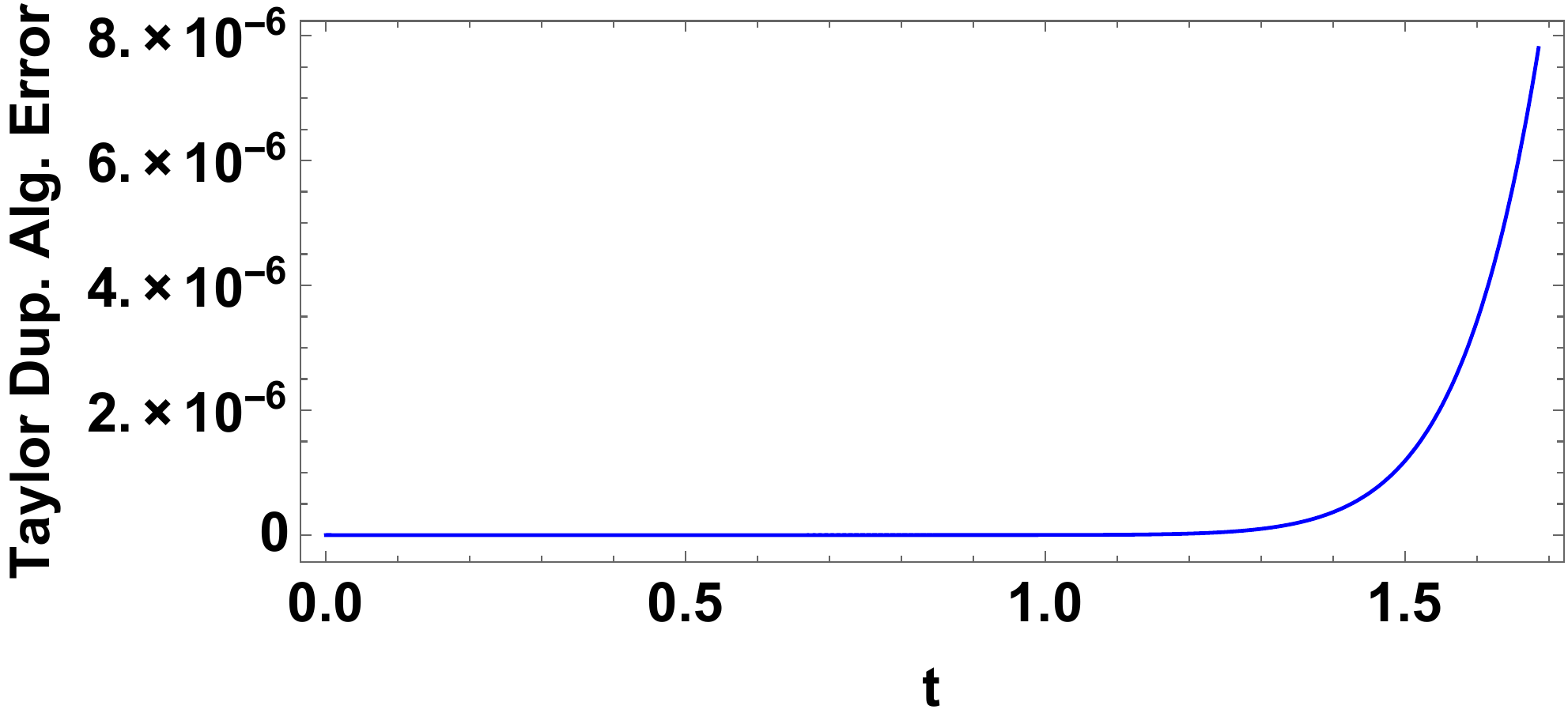}}\quad
\subfigure[Propagation of $\tilde{x}(t)-{x}(t)$. Computer time equal to $0.0191$ s. Taylor of order 15.]
{\label{DuplicationAprox}
\includegraphics[scale=0.35]{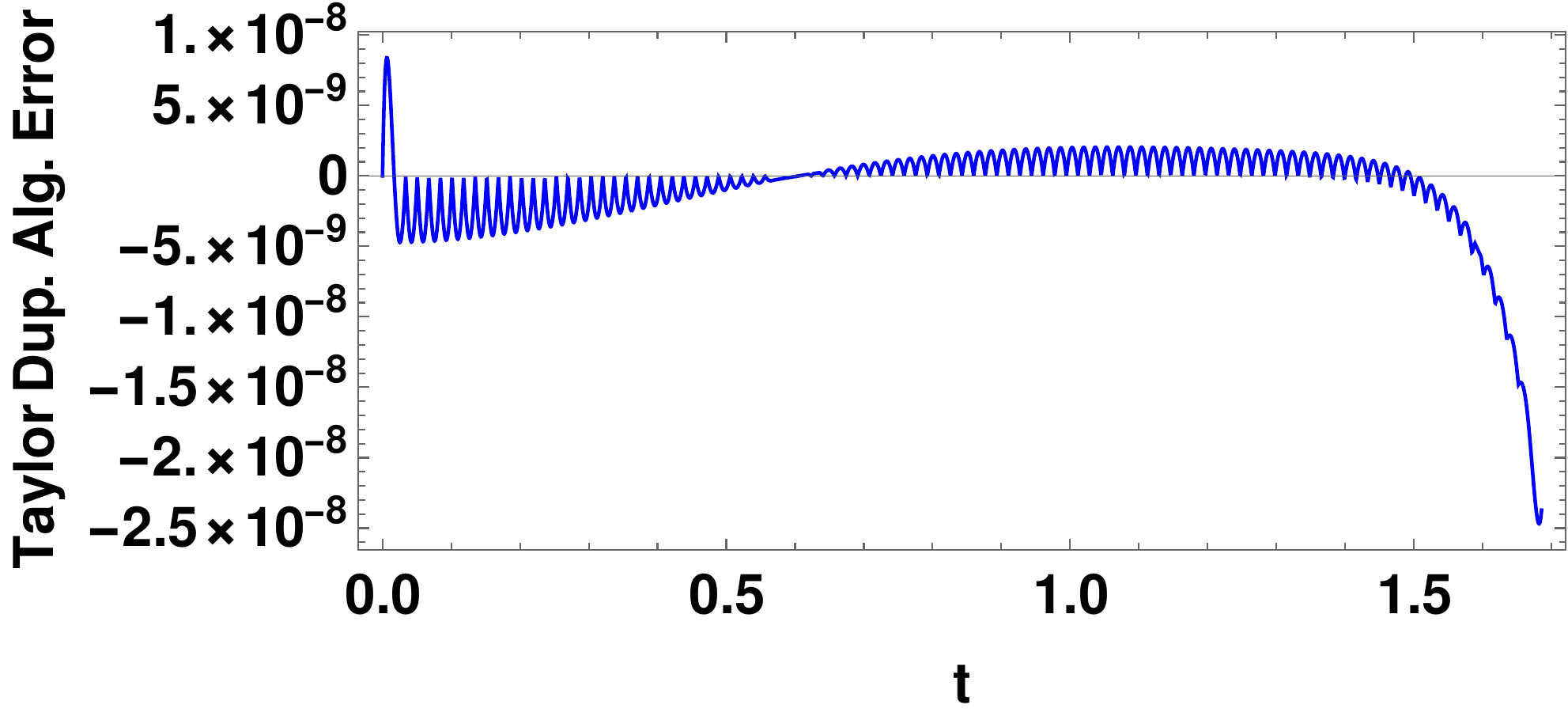}}
\subfigure[Propagation of $\tilde{x}(t)-{x}(t)$. Computer time equal to $0.0251$ s. Taylor of order 20.]
{\label{MathematicaAprox1}
\includegraphics[scale=0.35]{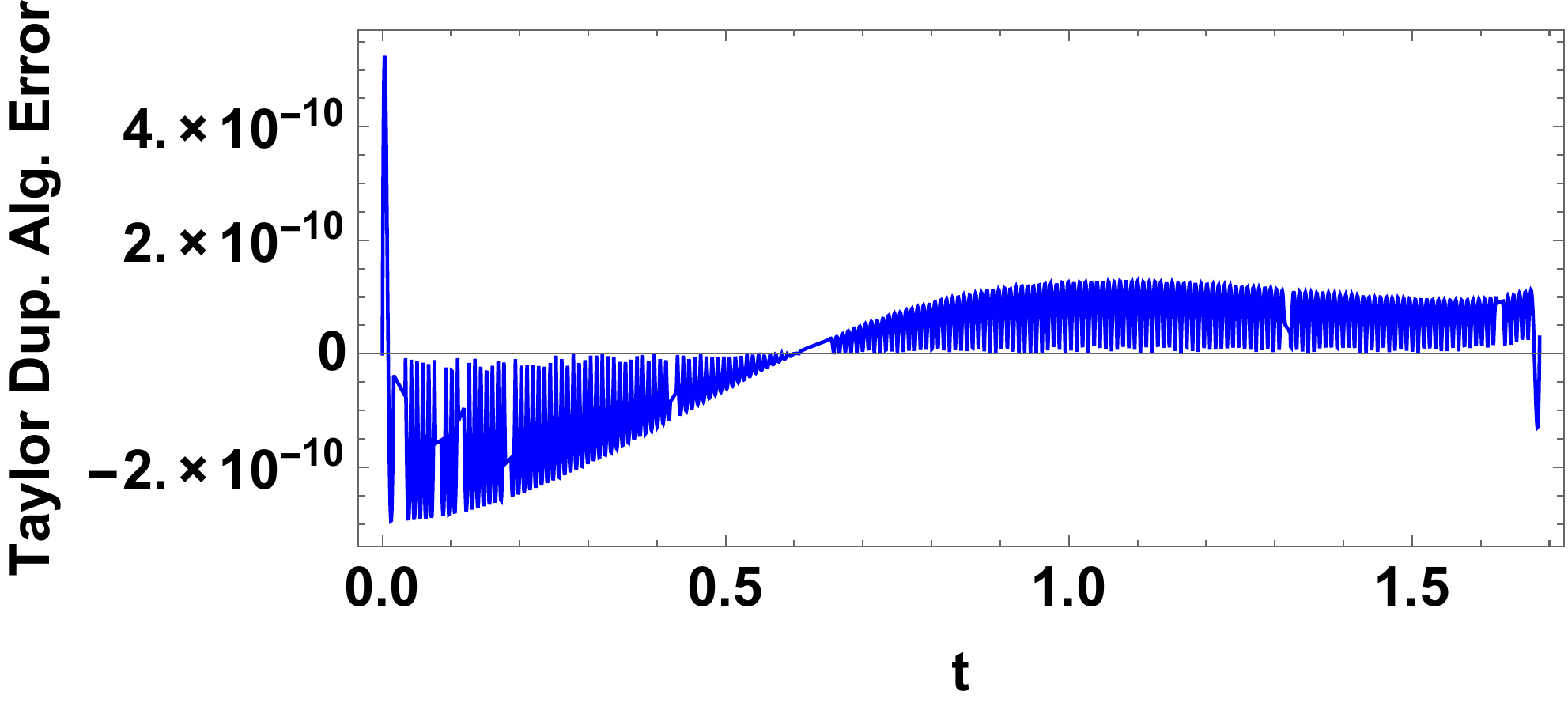}}\quad
\subfigure[Propagation of $\tilde{x}(t)-{x}(t)$. Computer time equal to $0.0323$ s. Taylor of order 30.]
{\label{DuplicationAprox1}
\includegraphics[scale=0.35]{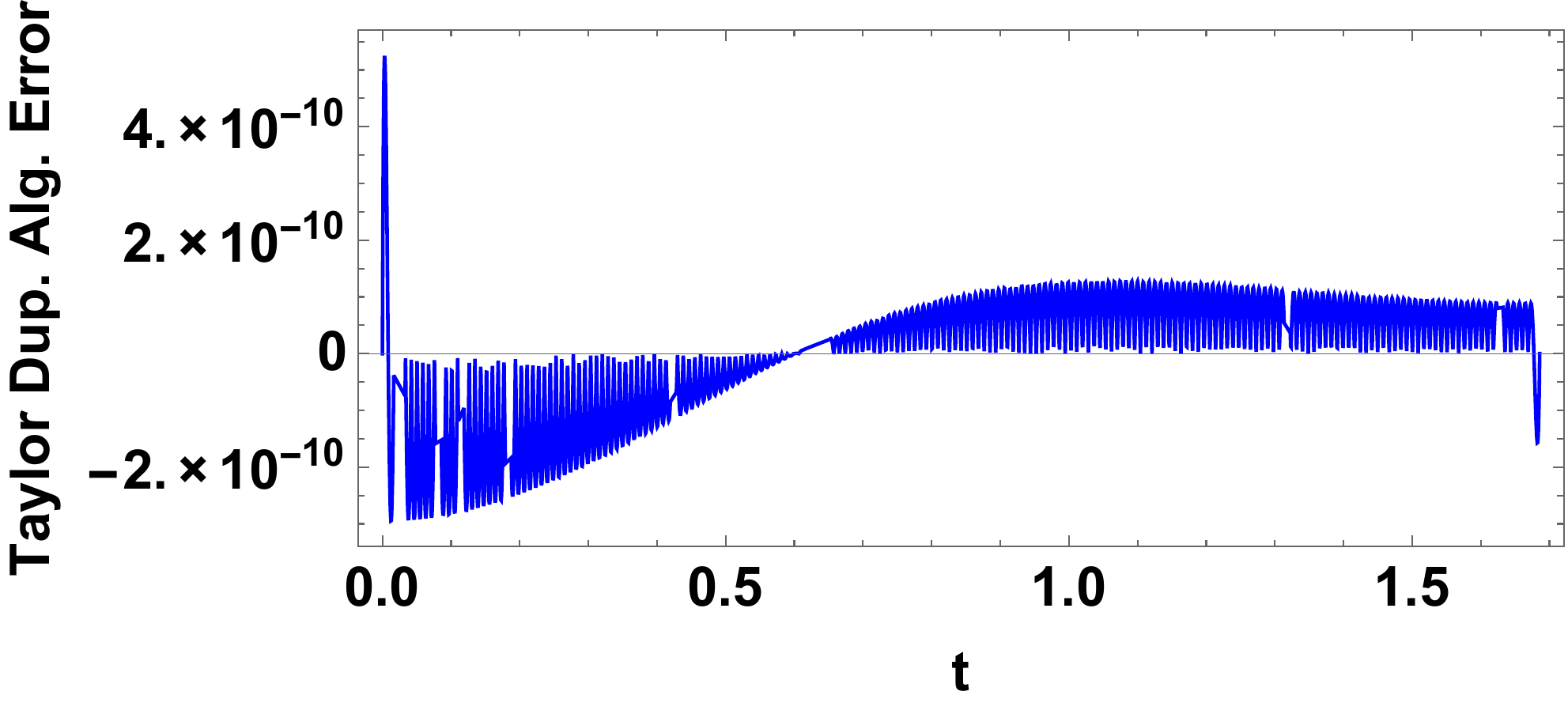}}
\caption{The influence of the order of approximation to the double-angle formula. The integration interval correspond to half a period of the function $x(t)$ and 200 interpolation points. Small parameters are set to $(\epsilon,\delta_\epsilon)= (2.5\times 10^{-31},10^{-15})$.}
\label{fig:ComparacionDupAlgWeierstrassOrdenTaylor}
\end{figure}

\begin{figure}[h!]
\subfigure[Propagation of $\tilde{x}(t)-{x}(t)$. Computer time equal to $0.0270$ s for 80 interpolation points.]
{\label{MathematicaAprox}
\includegraphics[scale=0.35]{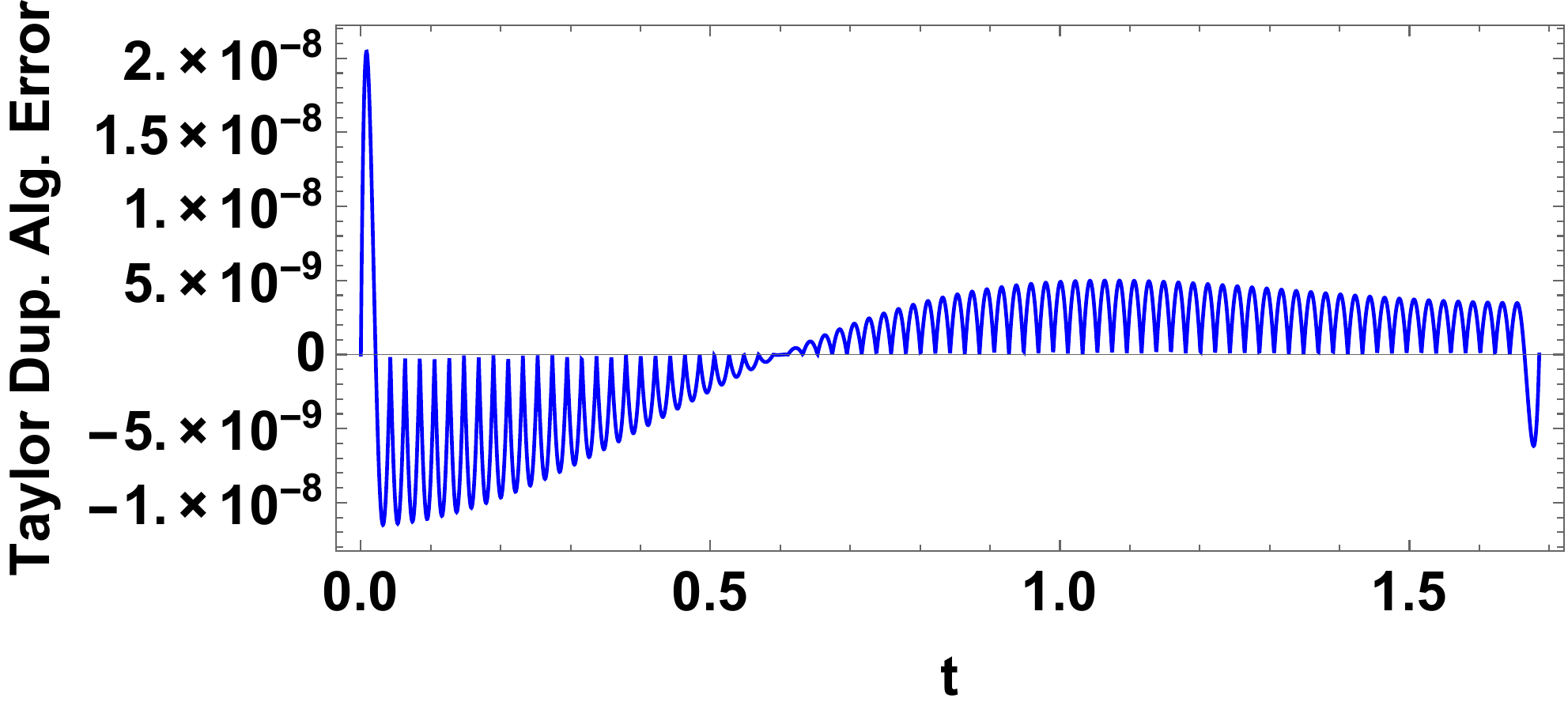}}\quad
\subfigure[Propagation of $\tilde{x}(t)-{x}(t)$. Computer time equal to $0.0534$ s for 160 interpolation points.]
{\label{DuplicationAprox}
\includegraphics[scale=0.35]{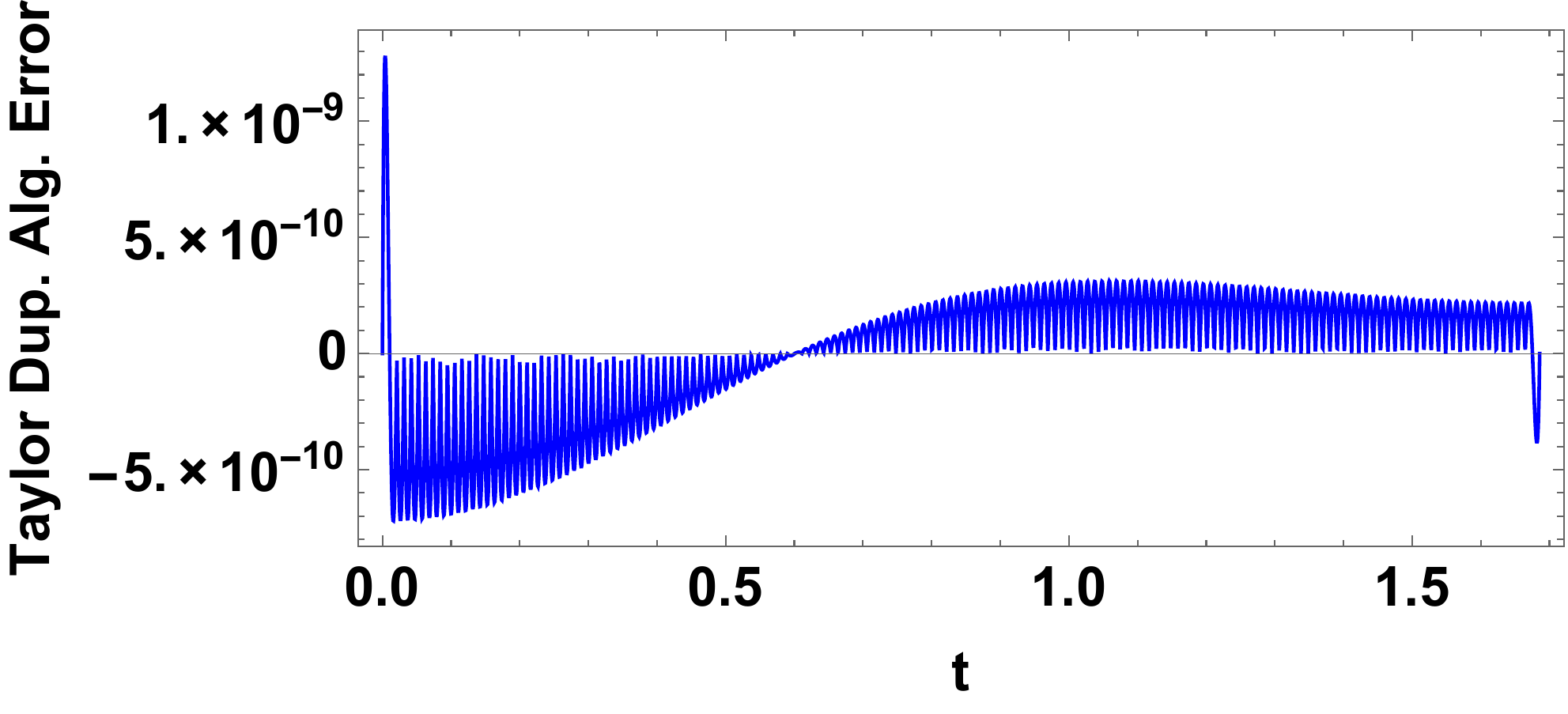}}\\
\subfigure[Propagation of $\tilde{x}(t)-{x}(t)$. Computer time equal to $0.0986$ s for 320 interpolation points.  ]
{\label{MathematicaAprox}
\includegraphics[scale=0.35]{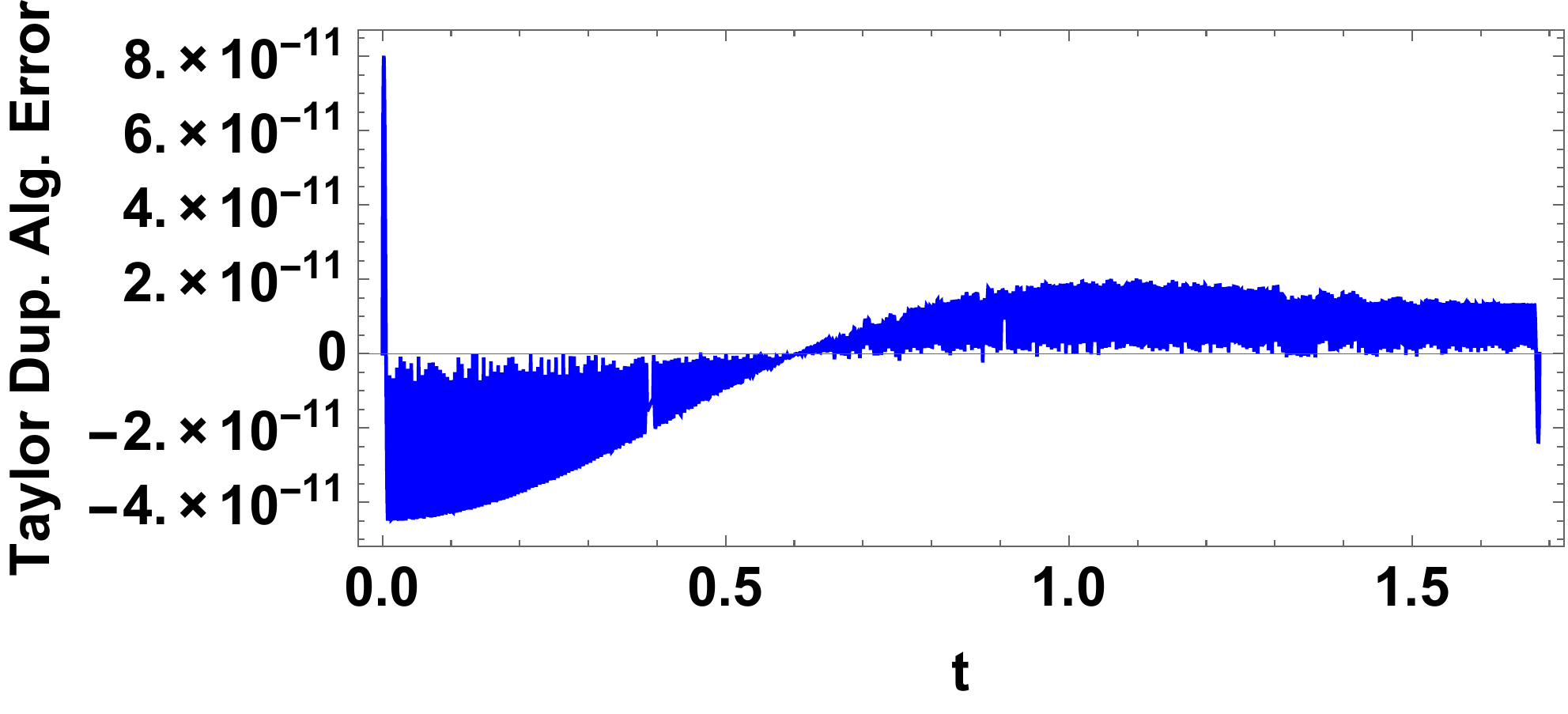}}\quad
\subfigure[Propagation of $\tilde{x}(t)-{x}(t)$. Computer time equal to $0.1799$ s for 640 interpolation points.]
{\label{DuplicationAprox}
\includegraphics[scale=0.35]{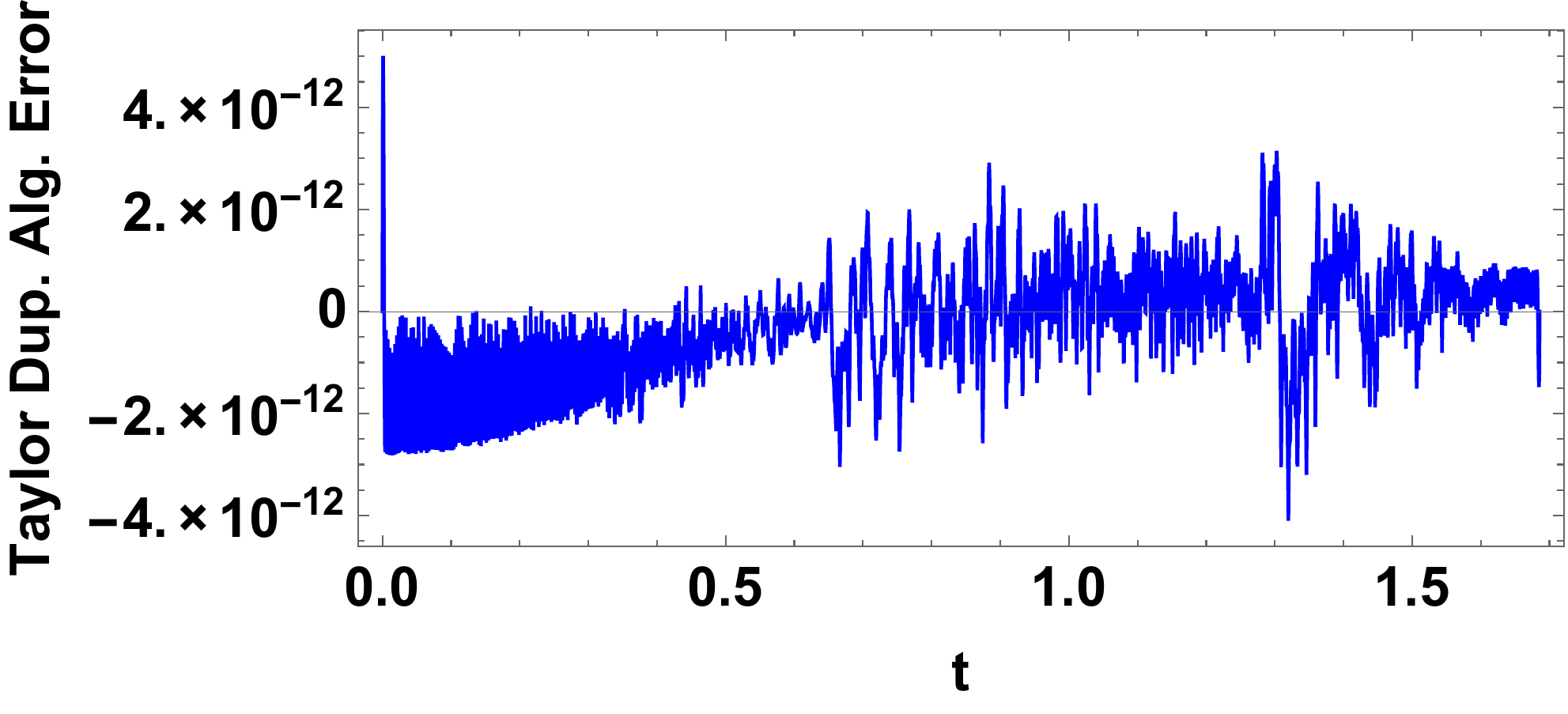}}\\
\subfigure[Propagation of $\tilde{x}(t)-{x}(t)$. Computer time equal to $0.203062$ s for 1280 interpolation points.]
{\label{DuplicationAprox}
\includegraphics[scale=0.35]{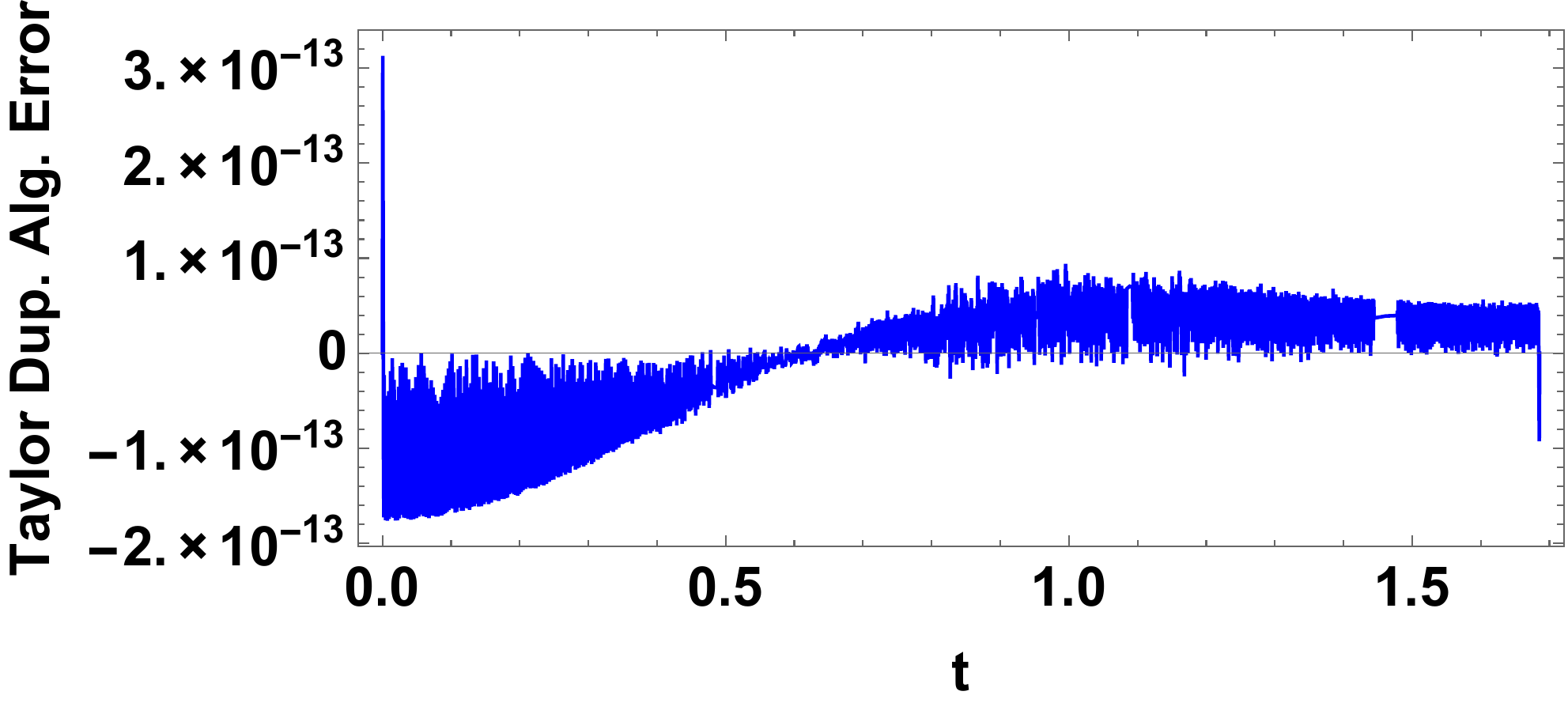}}\quad
\subfigure[Propagation of $\tilde{x}(t)-{x}(t)$. Computer time equal to $0.3693$ s for 2560 interpolation points.]
{\label{DuplicationAprox}
\includegraphics[scale=0.35]{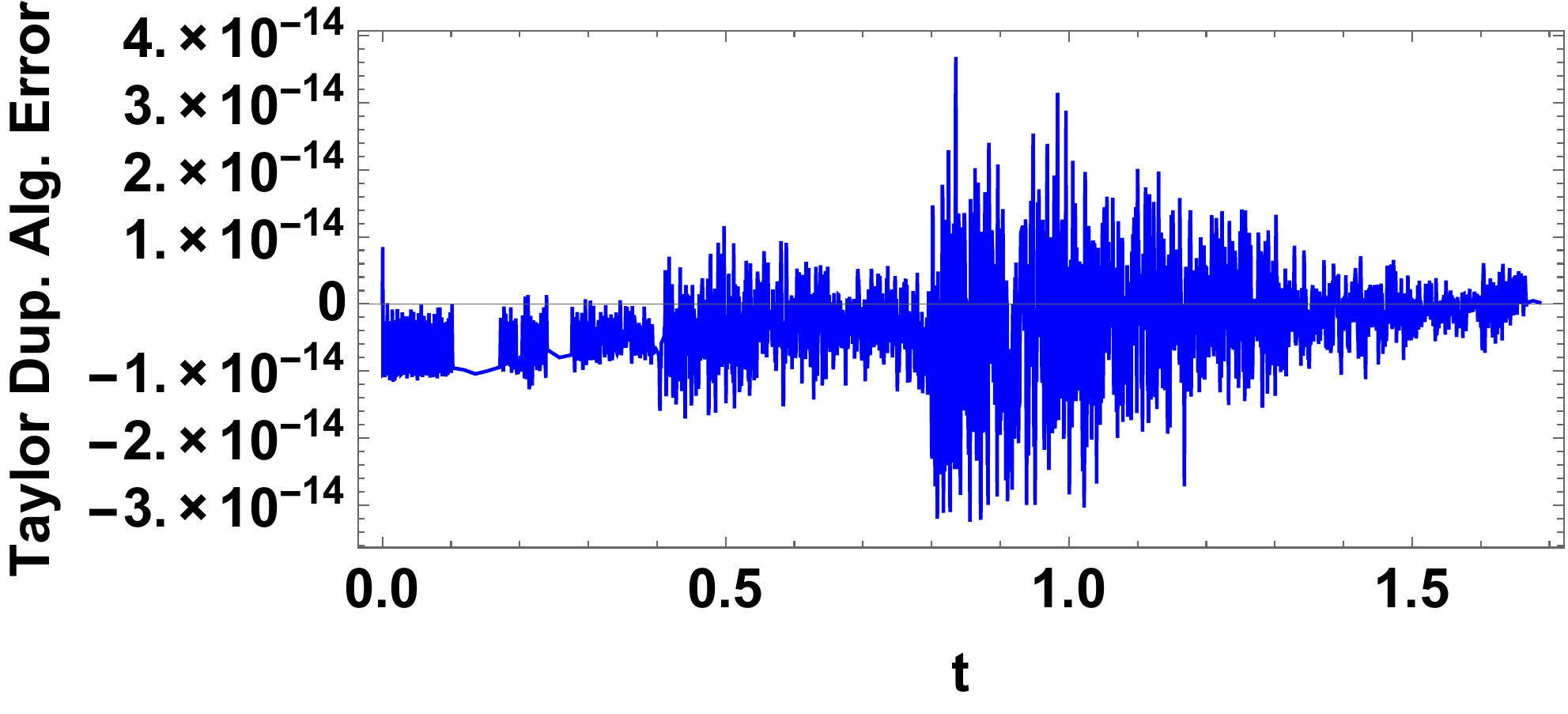}}
\caption{The role of the number of interpolation points. The double-angle formula and $x(t)$ have been approximated by a 30th order Taylor polynomial.  The integration interval correspond to half a period of the function $x(t)$. Small parameters are set to $(\epsilon,\delta_\epsilon)= (2.5\times 10^{-31},10^{-15})$.}
\label{fig:ComparacionMathematicaDupAlgWeierstrassAprox}
\end{figure}

\appendix
\section{Double-Angle Formula Taylor Expansion}
\label{sec:Appendix}
The Taylor expansion of the double-angle formula is obtained by following the strategy given in Section~\ref{sec:Computation}. Here we include the expression of the derivatives up to the 10th order
\begin{eqnarray*}
R(x_0)&=&x_0,\quad R^{(1)}(x_0)= 2,\quad R^{(2)}(x_0)=\frac{2f^{(1)}(x_0)}{f(x_0)},\quad R^{(3)}(x_0)= \frac{ 6  f^{(2)}(x_0)}{f(x_0)},
\end{eqnarray*}
\begin{eqnarray*}
R^{(4)}(x_0)&=& \frac{2}{f(x_0)^2}\left(7 f^{(3)}(x_0) f(x_0)+6 f^{(1)}(x_0)  f^{(2)}(x_0) \right),\\
\end{eqnarray*}
\begin{eqnarray*}
R^{(5)}(x_0)&=& \frac{30 \left(f(x_0) f^{(4)}(x_0)+2 f^{(2)}(x_0)^2+2 f^{(3)}(x_0) f^{(1)}(x_0)\right)}{f(x_0)^2},\\
\end{eqnarray*}
\begin{eqnarray*}
R^{(6)}(x_0)&=&\frac{1}{f(x_0)^3} \left(62 f^{(5)}(x_0) f(x_0)^2+20 \left(28 f(x_0) f^{(3)}(x_0) f^{(2)}(x_0)\right.\right.\\
&&\left.\left.+3 f^{(3)}(x_0) f^{(1)}(x_0)^2+f^{(1)}(x_0) \left(11 f(x_0) f^{(4)}(x_0)+9 f^{(2)}(x_0)^2\right)\right)\right),
\end{eqnarray*}
\begin{eqnarray*}
R^{(7)}(x_0)&=&\frac{14}{f(x_0)^3}
 \left(9 f^{(6)}(x_0) f(x_0)^2+10 \left(9 f^{(2)}(x_0)^3+4 f^{(4)}(x_0) f^{(1)}(x_0)^2\right.\right.\\
&&\left.\left. +5 f(x_0) \left(2 f^{(3)}(x_0)^2+f^{(5)}(x_0) f^{(1)}(x_0)\right)
\right.\right.\\
&&\left.\left. +\left(15 f(x_0) f^{(4)}(x_0)+23 f^{(3)}(x_0) f^{(1)}(x_0)\right) f^{(2)}(x_0)\right)\right),
\end{eqnarray*}
\begin{eqnarray*}
R^{(8)}(x_0)&=&\frac{1}{f(x_0)^4}\left(254 f^{(7)}(x_0) f(x_0)^3+28 \left(20 f^{(4)}(x_0) f^{(1)}(x_0)^3\right.\right.\\
&&\left.\left.
+f(x_0) \left(251 f(x_0) f^{(5)}(x_0) f^{(2)}(x_0)+5 f^{(3)}(x_0) \left(81 f(x_0) f^{(4)}(x_0)+199 f^{(2)}(x_0)^2\right)\right)
\right.\right.\\
&&\left.\left.+5 f^{(1)}(x_0)^2 \left(23 f(x_0) f^{(5)}(x_0)+32 f^{(3)}(x_0) f^{(2)}(x_0)\right)+f^{(1)}(x_0) \left(180 f^{(2)}(x_0)^3
\right.\right.\right.\\
&&\left.\left.\left.+f(x_0) \left(73 f(x_0) f^{(6)}(x_0)+455 f^{(3)}(x_0)^2\right)+715 f(x_0) f^{(4)}(x_0) f^{(2)}(x_0)\right)\right)\right),\\
\end{eqnarray*}
\begin{eqnarray*}
R^{(9)}(x_0)&=&\frac{1}{f(x_0)^4}\left(510 f^{(8)}(x_0) f(x_0)^3+84 \left(540 f^{(2)}(x_0)^4+90 f^{(5)}(x_0) f^{(1)}(x_0)^3
\right.\right.\\
&&\left.\left.+6 \left(29 f(x_0) f^{(6)}(x_0)+75 f^{(3)}(x_0)^2\right) f^{(1)}(x_0)^2
\right.\right.\\
&&\left.\left.+f(x_0)^2 \left(295 f^{(4)}(x_0)^2+488 f^{(3)}(x_0) f^{(5)}(x_0)\right)
\right.\right.\\
&&\left.\left.+f(x_0) \left(67 f(x_0) f^{(7)}(x_0)+1820 f^{(3)}(x_0) f^{(4)}(x_0)\right) f^{(1)}(x_0)
\right.\right.\\
&&\left.\left.+270 \left(7 f(x_0) f^{(4)}(x_0)+9 f^{(3)}(x_0) f^{(1)}(x_0)\right) f^{(2)}(x_0)^2
\right.\right.\\
&&\left.\left.+2 \left(405 f^{(4)}(x_0) f^{(1)}(x_0)^2+f(x_0) \left(130 f(x_0) f^{(6)}(x_0)
\right.\right.\right.\right.\\
&&\left.\left.\left.\left.+1235 f^{(3)}(x_0)^2+603 f^{(5)}(x_0) f^{(1)}(x_0)\right)\right) f^{(2)}(x_0)\right)\right),\\
\end{eqnarray*}
\begin{eqnarray*}
R^{(10)}(x_0)&=&\frac{2}{f(x_0)^5} \left(511 f^{(9)}(x_0) f(x_0)^4+3780 f^{(5)}(x_0) f^{(1)}(x_0)^4
\right.\\
&&\left.+252 f^{(1)}(x_0)^3 \left(117 f(x_0) f^{(6)}(x_0)+75 f^{(3)}(x_0)^2+175 f^{(4)}(x_0) f^{(2)}(x_0)\right)
\right.\\
&&\left.+12 f(x_0) \left(78330 f^{(3)}(x_0) f^{(2)}(x_0)^3+7 f(x_0) \left(3170 f^{(3)}(x_0)^3
\right.\right.\right.\\
&&\left.\left.\left.+4647 f^{(5)}(x_0) f^{(2)}(x_0)^2+14440 f^{(4)}(x_0) f^{(3)}(x_0) f^{(2)}(x_0)\right)
\right.\right.\\
&&\left.\left.
+f(x_0)^2 \left(2679 f^{(7)}(x_0) f^{(2)}(x_0)+5726 f^{(3)}(x_0) f^{(6)}(x_0)
\right.\right.\right.\\
&&\left.\left.\left.+8029 f^{(4)}(x_0) f^{(5)}(x_0)\right)\right)+42 f^{(1)}(x_0)^2 \left(7158 f(x_0) f^{(5)}(x_0) f^{(2)}(x_0)
\right.\right.\\
&&\left.\left.
+4350 f^{(3)}(x_0) f^{(2)}(x_0)^2+f(x_0) \left(683 f(x_0) f^{(7)}(x_0)+9530 f^{(3)}(x_0) f^{(4)}(x_0)\right)\right)
\right.\\
&&\left.+6 f^{(1)}(x_0) \left(18900 f^{(2)}(x_0)^4+158340 f(x_0) f^{(4)}(x_0) f^{(2)}(x_0)^2
\right.\right.\\
&&\left.\left.
+14 f(x_0) \left(2648 f(x_0) f^{(6)}(x_0)+13805 f^{(3)}(x_0)^2\right) f^{(2)}(x_0)
\right.\right.\\
&&\left.\left.+f(x_0)^2 \left(1237 f(x_0) f^{(8)}(x_0)+38115 f^{(4)}(x_0)^2+64974 f^{(3)}(x_0) f^{(5)}(x_0)\right)\right)\right).
\end{eqnarray*}

\section*{Acknowledgement}
Support from Research Agencies of Chile is acknowledged. They came in the form of research projects 11160224 of the Chilean national agency FONDECYT  and the UBB project 2020157 IF/R. The author J.L.Z. acknowledges support from CONICYT PhD/2017-21170836 and Proyecto Plurianual AIUE 1955 UBB.


\begin{thebibliography}{0}

\bibitem{Aczel1967}
J.~Acz\'el,
{\it Lectures on Functional Equations and Their Applications},
Mathematics in Science and Engineering, vol. 19, Academic Press, New York–London, 1966.

\bibitem{Bukhshtaber1993}
V.~M.~Bukhshtaber and I.~M.~Krichever, Vector addition theorems and Baker-Akhiezer functions.
\textit{ Theoretical and Mathematical Physics}, {\bf 94} (1993),142--149.

\bibitem{Bukhshtaber1996}
V.~M.~Bukhshtaber and I.~M.~Krichever,
Multidimensional vector addition theorems and the Riemann theta functions.
\textit{International Mathematics Research Notices}, {\bf 10} (1996),505--513.

\bibitem{Bulirsch1965}
R.~Bulirsch,
Numerical calculation of elliptic integrals and elliptic functions.
\textit{Numerische Mathematik}, {\bf 7} (1) (1965), 78--90.

\bibitem{Bulirsch1969}
R.~Bulirsch, Numerical calculation of elliptic integrals and elliptic functions {III}.
\textit{Numerische Mathematik}, {\bf 13} (4) (1969), 305--315.

\bibitem{Carlson1979}
B.~C. Carlson,
 Computing elliptic integrals by duplication.
\textit{Numerische Mathematik}, {\bf 33} (1) (1979),1--16.

\bibitem{Fukushima2015}
T.~Fukushima,
 Precise and fast computation of elliptic integrals and elliptic functions.
 In \textit{IEEE 22nd Symposium on Computer Arithmetic}, 2015.

\bibitem{Hancock}
H.~Hancock,
{\it Lectures on the Theory of Elliptic Functions}.
 John Wiley and Sons, 1910.
 
\bibitem{Krantz2002}
S.G.~Krantz and H.R.~Parks,
{\it A Primer of Real Analytic Functions}.
 Birkh\"auser Advanced Texts, 2002.

\bibitem{Kuwagaki1953}
A.~Kuwagaki,
 Sur la fonction analytique de deux variables complexes satisfaisant
  l'associativit{\'e} : $f\{x, f(y, z)\}=f\{f(x, y), z\}$.
\textit{Memoirs of the College of Science, University of Kyoto. Series
  A: Mathematics}, {\bf 27} (3) (1953), 225--234.

\bibitem{Lawden1989}
D.~F.~Lawden,
{\it Elliptic Functions and Applications},
  Applied Mathematical Sciences, Vol. 80.
 Springer-Verlag, 1989.

\bibitem{MontelPaul}
P.~Montel,
 Sur les fonctions d'une variable r\'eelle qui admettent un
  th\'eor\`eme d'addition alg\'ebrique.
\textit{Annales scientifiques de l'\'Ecole Normale Sup\'erieure}, 3e
  s{\'e}rie, {\bf 48} (1931), 65--94.

\bibitem{PaulPainleve}
P.~Painleve,
 Sur les fonctions qui admettent un th{\'e}or{\'e}me d'addition.
\textit{Acta Mathematica}, {\bf 26} (1902).

\bibitem{Phragmen1885}
E.~Phragmen,
Sur un th\'eor\'eme concernant les fonctions elliptiques.
\textit{Acta Mathematica}, {\bf 7} (1885), 33--42.

\bibitem{Ritt}
J.~F. Ritt,
Real functions with algebraic addition theorems.
\textit{Transactions of the American Mathematical Society},
  {\bf 29} (2) (1927), 361--368.

\bibitem{Tsiganov2009}
A.~V. Tsiganov, Leonard euler: Addition theorems and superintegrable systems.
\textit{Regular and Chaotic Dynamics}, {\bf 14} (3) (2009),389--406.

\bibitem{AbrahamUngar1983}
A.~Ungar,
 Addition theorems for solutions to linear homogeneous constant
  coefficient ordinary differential equations.
\textit{Aequationes Mathematicae}, {\bf 26} (1983), 104--112.

\bibitem{AbrahamUngar1987}
A.~Ungar, Addition theorems in ordinary differential equations.
\textit{The American Mathematical Monthly}, {\bf 94} (9) (1987), 872--875.

\bibitem{Weierstrass}
K.~Weierstrass,
Formeln uns Lehrstze sum Gebrauche des elliptischen Functionen. 1855.

\bibitem{ThesisZapata}
J.~Zapata,
Analytic and Geometric Techniques for Non-Integrable Dynamics
  Systems. Applications to Perturbed Keplerian Models.
 PhD thesis, Universidad del B\'{i}o-B\'{i}o, 2020.

\end{thebibliography}
\end{document}